\documentclass[a4paper,11pt]{article}
\usepackage{amsmath,amsthm,amssymb}
\usepackage[mathscr]{eucal}

\numberwithin{equation}{section}

{\theoremstyle{definition}\newtheorem{definition}{Definition}[section]

\newtheorem{remark}[definition]{Remark}
}

\newtheorem{proposition}[definition]{Proposition}
\newtheorem{lemma}[definition]{Lemma}
\newtheorem{theorem}[definition]{Theorem}
\newtheorem{corollary}[definition]{Corollary}

\newcommand{\GL}{\operatorname{GL}}
\newcommand{\Inn}{\operatorname{Inn}}

\newcommand{\dom}{\operatorname{dom}}
\newcommand{\bigfree}{\mathop{\text{\huge $\ast$}}}
\newcommand{\M}{\operatorname{M}}
\newcommand{\C}{\mathbb{C}}
\newcommand{\notembed}[1]{\underset{#1}{\not\prec}}
\newcommand{\embed}[1]{\underset{#1}{\prec}}

\newcommand{\F}{\mathbb{F}}
\newcommand{\cR}{\mathcal{R}}
\newcommand{\acts}{\curvearrowright}
\newcommand{\actson}[1][]{\overset{#1}{\curvearrowright}}
\newcommand{\SL}{\operatorname{SL}}
\newcommand{\rL}{\operatorname{L}}
\newcommand{\Aut}{\operatorname{Aut}}
\newcommand{\Out}{\operatorname{Out}}
\newcommand{\N}{\mathbb{N}}

\newcommand{\Z}{\mathbb{Z}}
\newcommand{\cF}{\mathcal{F}}

\newcommand{\si}{\sigma}

\newcommand{\recht}{\rightarrow}
\newcommand{\cU}{\mathcal{U}}
\newcommand{\vphi}{\varphi}

\newcommand{\R}{\mathbb{R}}
\newcommand{\al}{\alpha}
\newcommand{\eps}{\varepsilon}
\newcommand{\module}{\operatorname{mod}}
\newcommand{\Tr}{\operatorname{Tr}}
\newcommand{\ovt}{\overline{\otimes}}
\newcommand{\B}{\operatorname{B}}

\newcommand{\om}{\omega}
\newcommand{\cP}{\mathcal{P}}
\newcommand{\cZ}{\mathcal{Z}}
\newcommand{\Q}{\mathbb{Q}}

\newcommand{\Ker}{\operatorname{Ker}}

\newcommand{\Emb}{\operatorname{Emb}}

\newcommand{\Rp}{\R_+}

\newcommand{\ot}{\otimes}

\newcommand{\invlimit}{\mathop{\underleftarrow{\operatorname{lim}}}}

\newcommand{\Ad}{\operatorname{Ad}}
\newcommand{\Centr}{\operatorname{Centr}}

\newcommand{\be}{\beta}
\newcommand{\rC}{\operatorname{C}}

\newcommand{\bS}{\mathscr{S}}

\newcommand{\Sfactor}{\bS_{\text{\rm factor}}}
\newcommand{\Seqrel}{\bS_{\text{\rm eqrel}}}
\newcommand{\Scentr}{\bS_{\text{\rm centr}}}

\begin{document}
\begin{center}
{\LARGE\bf On the fundamental group of II$_1$ factors \vspace{0.5ex} and
equivalence relations arising from group actions}

\bigskip

{\sc by Sorin Popa\footnote{Partially supported by NSF Grant
DMS-0601082}\footnote{Mathematics Department; University of
    California at Los Angeles, CA 90095-1555 (United States).
    \\ E-mail: popa@math.ucla.edu} and Stefaan Vaes\footnote{Partially
    supported by ERC Starting Grant VNALG-200749 and Research
    Programme G.0231.07 of the Research Foundation --
    Flanders (FWO)}\footnote{Department of Mathematics;
    K.U.Leuven; Celestijnenlaan 200B; B--3001 Leuven (Belgium).
    \\ E-mail: stefaan.vaes@wis.kuleuven.be}}

\vspace{1ex}

{\it Dedicated to Alain Connes at the occasion of his 60th birthday.}
%
%
\end{center}

\begin{abstract}
\noindent Given a countable group $G$, we consider the sets
$\Sfactor(G)$, $\Seqrel(G)$, of subgroups $\cF\subset \R_+$ for which
there exists a free ergodic probability measure preserving action $G
\actson X$ such that the fundamental group of the associated II$_1$
factor $\rL^\infty(X)\rtimes G$, respectively orbit equivalence
relation $\cR(G\actson X)$, equals $\cF$. We prove that if
$G=\Gamma^{*\infty}* \Z$, with $\Gamma\neq 1$, then
$\Sfactor(G)$ and $\Seqrel(G)$ contain $\R_+$ itself, all of its
countable subgroups, as well as uncountable subgroups that can have
any Hausdorff dimension $\alpha \in (0,1)$. We deduce that there exist
II$_1$ factors of the form $M=\rL^\infty(X)\rtimes \F_\infty$ such
that the fundamental group of $M$ is $\R_+$, but $M \ovt
\B(\ell^2(\N))$ admits no continuous trace scaling action of $\R_+$.
We then prove that if $G=\Gamma * \Lambda$, with $\Gamma,
\Lambda$ finitely generated ICC groups, one of which has property
(T), then $\Sfactor(G)=\Seqrel(G)=\{\{1\}\}.$
\end{abstract}

\section{Introduction}

Some of the most intriguing phenomena concerning group measure space
II$_1$ factors $M=\rL^\infty(X)\rtimes G$ and orbit equivalence
relations $\cR=\cR(G\actson X)$, arising from free ergodic probability measure
preserving (p.m.p.) actions $G\actson X$ of countable groups $G$ on
probability spaces $(X,\mu)$, pertain to their {\it fundamental
groups} $\cF(M)$, $\cF(\cR)$ (\cite{MvN2}). Although much progress
has been made in understanding and calculating these invariants,
many natural questions on how the group $G$ may affect the behavior
of $\cF(M)$, $\cF(\cR)$ remain open.

A first indication that certain properties of $G$ can impact the
invariants independently of the way it acts appeared in Connes'
groundbreaking work on the classification and the structure of von
Neumann factors, from the 1970's. Thus, a side effect of the
uniqueness of the amenable II$_1$ factor \cite{connes1} and of the
amenable II$_1$ equivalence relation \cite{CFW}, is that
$\cF(M)=\cF(\cR)=\R_+$ whenever the group $G$ is amenable. On the
other hand, arguments from Connes' rigidity paper \cite{connes} were
used to show that if $G$ is infinite conjugacy class (ICC) and has
the property (T) of Kazhdan, then $\cF(M), \cF(\cR)$ are countable
for any free ergodic p.m.p.\ action of $G$ (\cite{Pcorr},
\cite{gg}).

Then, in the late 1990's, Gaboriau discovered that certain groups
$G$, such as the free groups with finitely many generators,
$\F_n, 2\leq n < \infty,$
give rise to orbit equivalence relations
$\cR=\cR(G\actson X)$ with
$\cF(\cR)=\{1\}$, for {\it any} free
ergodic p.m.p. action $F\actson X$ \cite{G1}. Moreover,
many factors of the form $\rL^\infty(X)\rtimes \F_n$ were shown
to have trivial fundamental group as well (cf.
\cite{P5}, \cite{ozpo}) and it is strongly believed that, in fact,
this holds true for all $\F_n \actson X$.

In turn, a completely new type of phenomena
emerged in the case $G=\F_\infty$, where it was shown
that  there exist free
ergodic p.m.p.\ actions $\F_\infty \actson X$ with the fundamental
group of the associated II$_1$ factors and equivalence relations,
$\cF(M), \cF(\cR)$, ranging over a ``large'' family of subgroups
$\cF\subset \R_+$, containing $ \R_+$ itself, all its countable
subgroups, as well as ``many'' uncountable subgroups $\neq \R_+$ \cite{PV1}.
In fact, it was conjectured in \cite{PV1} that any group $\cF$ that can be realized as a
fundamental group of a II$_1$ factor or equivalence relation can also
be realised as $\cF(\rL^\infty(X)\rtimes \F_\infty)$,
$\cF(\cR(\F_\infty \actson X))$, for
some free ergodic p.m.p. action $\F_\infty \actson X$.

Related to all these phenomena, we introduced in \cite{PV1} the sets
$\Sfactor(G)$, $\Seqrel(G)$, of subgroups $\cF \subset \Rp$ for
which there exists a free ergodic m.p.\ action $G \actson X$ such
that $\cF(\rL^\infty(X)\rtimes G)=\cF$, respectively
$\cF(\cR(G\actson X))=\cF$. Using this notation, the result in
\cite{PV1} shows, more precisely, that $ \Sfactor(\F_\infty)\cap
\Seqrel(\F_\infty)$ contains the set $\Scentr$ of all subgroups
$\cF\subset \Rp$ for which there exists a free ergodic action of an
amenable group $\Lambda$ on an infinite measure space $(Y,\nu)$,
such that the set of scalars $t>0$ that can appear as scaling
constants of non-singular automorphisms $\theta$ of $(Y,\nu)$
commuting with $\Lambda \actson Y$ equals $\cF$. In turn, $\Scentr$
is shown to contain $\R_+$, all its countable subgroups and
uncountable subgroups $\cF\subset \R_+$ with arbitrary Hausdorff dimension in the interval
$(0,1)$ (\cite{PV1}). While an abstract characterization of
$\Sfactor(\F_\infty)$, $\Seqrel(\F_\infty)$ remains elusive, it was
noticed in \cite{PV1} that subgroups in either set, as in fact
subgroups in $\Sfactor(G)$, $\Seqrel(G)$ for any $G$, must be Borel
sets and Polishable.

Our purpose in this paper is to estimate (or even completely
calculate) the invariants $\Sfactor(G)$, $\Seqrel(G)$ for other
classes of groups $G$. We target two types of results: on the
one hand, detecting classes of groups $G$ for which $\Sfactor(G)$,
$\Seqrel(G)$ are ``large'', containing for instance the set
$\Scentr$ defined above (like in the case case $G=\F_\infty$); on
the other hand, detecting classes of groups $G$ for which
$\Sfactor(G)$, $\Seqrel(G)$ contain only ``small'' subgroups of
$\Rp$ (e.g. countable, or just $\{1\}$).

Thus, our first result enlarges considerably the class of groups $G$
for which we can show that the set $\Scentr$ is contained in both
$\Sfactor(G)$ and $\Seqrel(G)$.

\begin{theorem} \label{thm0.1.Sfactorlarge}
Let $\Gamma$ be a non-trivial group, $\Sigma$ an infinite
amenable group and denote $G=\Gamma^{*\infty} * \Sigma$. Then,
$$\Scentr \subset \Sfactor(G)
\qquad\text{and}\qquad \Scentr \subset \Seqrel(G) \; .$$

Moreover, there exist free ergodic p.m.p.\ actions $G \actson
(X,\mu)$ such that the II$_1$ factor $M = \rL^\infty(X) \rtimes G$
has fundamental group $\cF(M)= \Rp$, but the II$_\infty$ factor $M
\ovt \B(\ell^2(\N))$ admits no trace scaling action of $\Rp$.
\end{theorem}

In Section \ref{sec.countable}, we will show that if
the full group of an equivalence relation $\cR$ on
a probability space $(X,\mu)$ contains a property (T) group acting
ergodically on $X$, then $\cF(\cR)$ is countable. Thus, if a
group $\Gamma$ appearing in Theorem \ref{thm0.1.Sfactorlarge} contains an infinite
subgroup $\Lambda$ with the property (T) and if $G=\Gamma^{*\infty}
* \Sigma \curvearrowright X$ is a free ergodic p.m.p.\ action such
that $\cR_G$ has fundamental group equal to an uncountable group in
$\Scentr$, then the restriction of $G \curvearrowright X$ to
$\Lambda$ cannot be ergodic.

Note that the last part of Theorem \ref{thm0.1.Sfactorlarge} provides group
measure space II$_1$ factors $M=\rL^\infty(X)\rtimes G$ which do
have fundamental group equal to $\R_+$ yet cannot appear in the
continuous decomposition of a type III$_1$ factor. The problem of
whether such II$_1$ factors exist was posed over the years by
several people, including Connes, Takesaki, and more recently
Shlyakhtenko. The fact that there are even factors of the form
$\rL^\infty(X)\rtimes \F_\infty$ satisfying this property (by simply
taking $\Gamma=\Sigma=\Z$ in 1.1) should be contrasted with the
fact that the II$_\infty$ factor associated with $L(\F_\infty)$ does
admit a trace scaling action of $\R_+$, by \cite{rad}.

Note that all groups of the form $G=\Gamma^{*\infty}*\Sigma$,
covered by the above theorem, have infinite first $\ell^2$-Betti
number, $\beta^{(2)}_1(G)=\infty$, and in fact
$\beta^{(2)}_n(G)=\infty, 0$, $\forall n\geq 2$. On the other hand,
by Gaboriau's scaling formula for $\ell^2$-Betti numbers \cite{G1},
any free ergodic p.m.p.\ action of a group $G$ with
$\beta^{(2)}_n(G)\neq 0, \infty$, for some $n$, gives rise to an
equivalence relation $\cR_G$ with trivial fundamental group,
$\cF(\cR_G)=\{1\}$. In other words, $\Seqrel(G) = \{\{1\}\}$. While
it is still an open question whether the corresponding II$_1$
factors $M=\rL^\infty(X)\rtimes G$ satisfy $\cF(M)=\{1\}$ as well
(i.e. $\Sfactor(G) = \{\{1\}\}$), our next result provides a large
class of groups $G$ for which this is indeed the case.

\begin{theorem} \label{thm0.2.Sfactortrivial}
Let $\Gamma$ and $\Lambda$ be infinite, finitely generated groups.
Assume that $\Gamma$ is ICC and that one of the following conditions
holds.
\begin{enumerate} \renewcommand{\theenumi}{\alph{enumi}}
\renewcommand{\labelenumi}{\theenumi)}
\item $\Gamma = \Gamma_1 \times \Gamma_2$,
with $\Gamma_1$ non-trivial and $\Gamma_2$ non-amenable,
\item $\Gamma$ admits a non-virtually abelian,
normal subgroup $\Gamma_1$ with the relative property (T).
\end{enumerate}
Then, $\Sfactor(\Gamma * \Lambda) = \Seqrel(\Gamma*\Lambda) =
\{\{1\}\}$.
\end{theorem}

When viewed from the perspective of Connes' discrete decomposition
of type III$_\lambda$ factors with $0<\lambda <1$ (\cite{connes2})
and respectively Connes-Takesaki continuous decomposition of type
III$_1$ factors (\cite{CT}), the above result provides a large class
of groups $G$ with the property that no II$_1$ factor $M$ arising
from an arbitrary free ergodic p.m.p.\ action of $G$ can appear in
the decomposition of a type III factor (i.e., as Connes puts it, no
such $M$ can appear as the ``shadow'' of a type III factor).

While II$_1$ factors $M=\rL^\infty(X)\rtimes \Gamma$ arising from
free ergodic p.m.p.\ actions $\Gamma \curvearrowright X$ of ICC
property (T) groups always have countable fundamental group (cf.
\cite{connes}, \cite{Pcorr}, \cite{gg}), it was not known whether
there exist cases when $\cF(M)\neq \{1\}$. Our next result gives the
first such examples. It also provides the first ``concrete''
examples of free ergodic p.m.p.\ actions $\Gamma \curvearrowright X$
with the associated II$_1$ factors $M$ having fundamental group
$\neq \{1\},
\R_+$. Indeed, the actions in Theorem 1.1 above and in
\cite{PV1} are shown to exist by using a Baire-category argument, at some point,
while in 1.3 below they are specific $G$-actions, obtained as
diagonal products of Bernoulli and profinite actions.

\begin{theorem} \label{thm0.5.rational}
Let $\cF \subset \Q_+$ be a subgroup generated by a subset of the
prime numbers. Let $G = \Z^n \rtimes \SL(n,\Z)$ with $n \geq 3$.
Then $G$ admits a free ergodic p.m.p.\ action $G \actson (X,\mu)$
such that the fundamental group of $\rL^\infty(X) \rtimes G$ and of
$\cR(G \actson X)$ equals $\cF$.
\end{theorem}

We in fact believe that any subgroup of $\Q_+$ can be realized as
the fundamental group of a factor or equivalence relation arising
from a free ergodic p.m.p.\ action of $\Z^n \rtimes \SL(n,\Z)$, $n
\geq 3$. The question of whether there exist free ergodic p.m.p.\
actions of an ICC property (T) group $G \curvearrowright X$ such that
$\cF(\cR_G)$ or $\cF(\rL^\infty(X)\rtimes G)$
contains irrational numbers remains open. In fact, it is not
even known whether
the union of all the fundamental groups
of II$_1$ factors and equivalence relations arising
from free ergodic p.m.p.\ actions of a fixed ICC property (T)
group $G$ is necessarily countable or not.

Finally, noticing that for a large number of groups $G$ it is known
that $\{1\}\in \Sfactor(G)$ (see e.g. \cite{P2}, \cite{P5},
\cite{Picm}), we conjecture that this is in fact the case for all
non-amenable groups $G$. If true, this would also show that the only
possibilities for $\Sfactor(G), \Seqrel(G)$ to be single point sets
are $\Sfactor(G)= \Seqrel(G)=\{\R_+\}$, $\Sfactor(G)=
\Seqrel(G)=\{\{1\}\}$, the first situation corresponding to $G$
being amenable. This would provide a new, interesting facet of the
dichotomy amenable/non-amenable for groups.

\section{Preliminaries}

The \emph{fundamental group} $\cF(M)$ of a II$_1$ factor $M$,
introduced in \cite{MvN2}, is defined as the following subgroup of $\R_+$.
$$\cF(M) = \{ \tau(p)/\tau(q) \mid p,q \;\;\text{are non-zero
projections in $M$ such that}\;\; pMp \cong qMq \} \; .$$
We call II$_1$ equivalence relation on a standard probability
space $(X,\mu)$ any ergodic probability measure preserving (p.m.p.)
measurable equivalence relation with countable equivalence classes.
The fundamental group $\cF(\cR)$ of a II$_1$ equivalence relation $\cR$
is defined as
$$\cF(\cR) = \{ \mu(Y)/\mu(Z) \mid \cR|_Y \cong \cR|_Z \} \; .$$
Whenever $\Gamma \actson (X,\mu)$ is a free ergodic p.m.p.\ action,
we denote by $\cR(\Gamma \actson X)$ the associated II$_1$ orbit
equivalence (OE) relation and by $\rL^\infty(X) \rtimes \Gamma$ the
associated group measure space II$_1$ factor \cite{MvN2}.

\begin{definition} \label{def.rigid}
A free ergodic p.m.p.\ action $\Gamma \actson (X,\mu)$ is called
\emph{rigid} if the corresponding inclusion $\rL^\infty(X) \subset
\rL^\infty(X) \rtimes \Gamma$ is rigid in the sense of \cite[Proposition 4.1]{P5}.
\end{definition}

\subsection*{Some sets of subgroups of $\R$ and ergodic measures}

Given a countable group $\Gamma$, we are interested in
\begin{align*}
\Sfactor(\Gamma) & := \{ \cF \subset \Rp \mid \;  \text{there exists
a free ergodic p.m.p.\ action $\Gamma \actson (X,\mu)$}\\ &
\hspace{2.4cm}\text{such that}\; \cF(\rL^\infty(X) \rtimes \Gamma) = \cF \} \; , \\
\Seqrel(\Gamma) & := \{ \cF \subset \Rp \mid \;  \text{there exists a
free ergodic p.m.p.\ action $\Gamma \actson (X,\mu)$}\\ &
\hspace{2.4cm}\text{such that}\; \cF(\cR(\Gamma \actson X)) =
\cF \} \; .
\end{align*}

In \cite[Theorem 5.3 and formula (2.2)]{PV1}, we have shown that
both $\Sfactor(\F_\infty)$ and $\Seqrel(\F_\infty)$ contain
$\Scentr$, defined as
\begin{align*}
\Scentr & := \{ \cF \subset \Rp \mid \; \text{there exists $\; \Lambda
\actson (Y,\eta) \;$, a free ergodic m.p.\ action,} \\ & \hspace{2.4cm}\text{with
$\Lambda$ amenable and $\; \module(\Centr_\Lambda(Y)) = \cF$} \; \} \; .
\end{align*}

Following \cite[Section 4]{A}, we call \emph{ergodic measure on $\R$}
any $\sigma$-finite measure $\nu$ on the Borel sets of $\R$ having
the following properties, where we denote $\lambda_x(y) = x+y$.
\begin{itemize}
\item For all $x \in \R$, either $\nu \circ \lambda_x = \nu$ or
$\nu \circ \lambda_x \perp \nu$.
\item There exists a countable subgroup $Q \subset \R$ such that
$\nu \circ \lambda_x = \nu$ for all $x \in Q$
and such that every $Q$-invariant Borel function on $\R$ is
$\nu$-almost everywhere constant.
\end{itemize}
For every ergodic measure $\nu$ on $\R$, one defines
$$H_\nu := \{x \in \R \mid \nu \circ \lambda_x = \nu \} \; .$$
As shown in \cite{A},
the groups $H_\nu$ can have arbitrary Hausdorff dimension and all $\exp(H_\nu)$
belong to $\Scentr$. We refer to \cite[Section 2 and the proof of Theorem 5.3]{PV1}
for a detailed exposition.

\subsection*{Intertwining by bimodules and the notation $A \underset{M}{\prec} B$ } 			

In Sections \ref{sec.Sfactorlarge} and \ref{sec.Sfactortrivial},
we use the method of intertwining by bimodules, introduced by the
first author in \cite{P5}. Let $(M,\tau)$ be a von Neumann algebra
with faithful normal tracial state $\tau$. We use the notation
$M^n = \M_n(\C) \ot M$.
When $A,B \subset M^n$ are possibly non-unital embeddings,
we write $A \embed{M} B$ if there exists a non-zero partial
isometry $v \in 1_A (\M_{n,m}(\C) \ot M)$ and a, possibly
non-unital, normal $^*$-homomorphism $\rho : A \recht B^m$
satisfying $a v = v \rho(a)$ for all $a \in A$. Several
equivalent formulations of this property can be given;
see \cite[Theorem 2.1]{P1} (see also \cite[Theorem C.3]{V1}).

Suppose that $A$ and $B$ are Cartan subalgebras of the II$_1$
factor $M$. Let $A_0 \subset A$ be a von Neumann subalgebra such
that $A_0' \cap M = A$. By \cite[Theorem A.1]{P5}, $A_0 \embed{M} B$
if and only if there exists a unitary $u \in M$ such that $uAu^* = B$.

\section{Groups $G$ for which $\Sfactor(G)$ contains uncountable groups}
\label{sec.Sfactorlarge}

The following theorem, whose proof is given at the end of the
section, provides a large family of groups $G$ such that
$\Sfactor(G)$ and $\Seqrel(G)$ is large, in the sense that both
contain $\Scentr$. Moreover, we prove that $G$ admits free ergodic p.m.p.\ actions $G \actson (X,\mu)$ such that the II$_1$ factor $M := \rL^\infty(X) \rtimes G$ has fundamental group $\R_+$, but nevertheless, the II$_\infty$ factor $M \ovt \B(\ell^2(\N))$ admits no strongly continuous trace scaling action of $\R_+$.

The groups $G$ involved are infinite free product groups
and should be contrasted with the groups $G$ treated in Theorem
\ref{thm.Sfactortrivial}, for which $\Sfactor(G)$ is trivial (cf.\
Remark \ref{rem.smallvslarge}).

\begin{theorem} \label{thm.Sfactorlarge}
Let $\Gamma$ be a non-trivial group, $\Sigma$ an infinite
amenable group and denote $G=\Gamma^{*\infty} * \Sigma$. Then,
$$\Scentr \subset \Sfactor(G)
\qquad\text{and}\qquad \Scentr \subset \Seqrel(G) \; .$$
Moreover, there exist free ergodic p.m.p.\ actions $G \actson
(X,\mu)$ such that the II$_1$ factor $M = \rL^\infty(X) \rtimes G$
has fundamental group $\cF(M)= \Rp$, but the II$_\infty$ factor $M
\ovt \B(\ell^2(\N))$ admits no trace scaling action of $\Rp$.
\end{theorem}

In the course of the proof of Theorem \ref{thm.Sfactorlarge}, we will also
obtain the following result.

\begin{theorem} \label{thm.tensor-product}
There exist II$_1$ factors $M_1$ and $M_2$ such that $\cF(M_1) \neq \Rp \neq \cF(M_2)$,
but nevertheless $\cF(M_1 \ovt M_2) = \Rp$.
\end{theorem}

Let $G$ be a countable group with subgroup $\Gamma$. Suppose that $G \actson (X,\mu)$
is a free p.m.p.\ action such that the restriction to $\Gamma$ is ergodic.
Slightly changing notations compared to \cite[Section 2]{PV1}, denote by $\Emb(\Gamma,G)$
the set of non-singular partial automorphisms $\phi$ of $(X,\mu)$ satisfying
$\phi(g \cdot x) \in G \cdot \phi(x)$ for all $g \in \Gamma$ and almost all $x \in X$
with $x,g \cdot x \in D(\phi)$. Denote by $[[G]]$ the full pseudogroup of the
OE relation $\cR(G \actson X)$, i.e.\ the set of a partial automorphisms $\phi$
of $(X,\mu)$ satisfying $\phi(x) \in G \cdot x$ for almost all $x \in D(\phi)$.

The following lemma generalizes \cite[Theorem 4.1]{PV1}.

\begin{lemma} \label{lemma.nosymmetry}
Let $\Gamma$ be an infinite group, $\Lambda$ an arbitrary group, both acting freely
and p.m.p.\ on $(X,\mu)$. There exists a free p.m.p.\ action
$\Gamma^{*\infty} * \Lambda \actson[\al] (X,\mu)$ with the following properties.
\begin{itemize}
\item The restriction of $\al$ to $\Gamma^{*\infty}$ is ergodic and rigid
(in the sense of Definition \ref{def.rigid}).
\item $\Emb(\Gamma^{*\infty},\Gamma^{*\infty} * \Lambda) =
[[\Gamma^{*\infty} * \Lambda]]$.
\item The restriction of $\al$ to any of the copies of $\Gamma$,
resp.\ to $\Lambda$, is conjugate to the originally given action.
\end{itemize}
\end{lemma}

\begin{proof}
Denote the given actions by $\Gamma \actson[\beta] (X,\mu)$ and
$\Lambda \actson[\rho] (X,\mu)$. We introduce the following notations:
\begin{align*}
\Gamma^{*\infty} = & \ \bigfree_{n=-1}^\infty G_n \quad\text{with all}\;\;
G_n \cong \Gamma \; , \\
\Gamma_n := & \ \bigfree_{k=-1}^n G_k \; , \\ \Gamma_E := & \ G_{-1} * G_0 *
\bigfree_{n \in E} G_n \quad\text{whenever}\;\; E \subset \N \; .
\end{align*}
By \cite[Theorem 1.2]{G2}, take a free ergodic p.m.p.\ action
$\Gamma_0 \actson[\al_0] (X,\mu)$ such that $\al_0$ is a rigid
action and such that the restrictions of $\al_0$ to $G_{-1}$ and $G_0$ are
conjugate to the action $\beta$. By \cite[Category Lemma]{To} and
\cite[Lemma A.1]{IPP}, extend $\al_0$ to a free action of $\Gamma_0 *
\Lambda$ on $(X,\mu)$, still denoted by $\al_0$, whose restriction to
$\Lambda$ is conjugate to the action $\rho$.

Extend the action $\al_0$ inductively to free actions $\Gamma_n *
\Lambda \actson[\al_n] (X,\mu)$ following the procedure in \cite[Section 3]{PV1}
and such that the restriction of $\al_n$ to $G_k \subset \Gamma_n$ is conjugate
to $\beta$ for all $k \leq n$. We end up with the free action
$\Gamma^{*\infty} * \Lambda \actson[\al_\infty] (X,\mu)$.
For every infinite subset $E \subset \N$, we denote by $\al_E$
the restriction of $\al_\infty$ to $\Gamma_E * \Lambda$.
Following the proof of \cite[Theorem 4.1]{PV1}, there exists an infinite
subset $E \subset \N$ such that $\Emb(\Gamma_E,\Gamma_E * \Lambda) = [[\Gamma_E * \Lambda]]$.
Since $\Gamma_E \cong \Gamma^{*\infty}$, the lemma is proved.
\end{proof}

\begin{remark}
Using the methods of \cite[Section 2.3]{G2}, Lemma \ref{lemma.nosymmetry}
can be shown for $\Gamma_1 * \Gamma_2$ instead of the infinite free product
$\Gamma^{*\infty}$, for arbitrary infinite groups $\Gamma_1,\Gamma_2$ with
given free p.m.p.\ actions on $(X,\mu)$. Such a generalization does not provide
a refinement for Theorem \ref{thm.Sfactorlarge} though, since the proof of
Theorem \ref{thm.Sfactorlarge} involves taking once more an infinite free product.
\end{remark}

For the formulation of the following theorem, recall that the automorphism group $\Aut(N)$ of a von Neumann algebra with separable
predual is a Polish group under the topology making the maps $\Aut(N) \recht N_* :
\al \mapsto \om \circ \al$ continuous for all $\om \in N_*$. Similarly, the group $\Aut(Y,\eta)$ of non-singular isomorphisms of $(Y,\eta)$ (up to equality almost everywhere) is a Polish group and $\Centr_{\Aut Y}(\Gamma_2)$ is a closed subgroup whenever $\Gamma_2 \actson (Y,\eta)$ is a non-singular action.

\begin{theorem} \label{thm.compute-out}
Let $\Gamma_1 * \Gamma_2 \actson[\al] (X,\mu)$ be a free p.m.p.\ action.
Let $\Gamma_2 \actson (Y,\eta)$ be a free ergodic action preserving the
infinite standard measure $\eta$. Consider the action
$\Gamma_1 * \Gamma_2 \actson X \times Y$ given by
\begin{equation} \label{eq.action}
g \cdot (x,y) = (g \cdot x,y) \;\;\forall g \in \Gamma_1 \;\; , \;\; h \cdot (x,y) =
(h \cdot x ,h \cdot y) \;\;\forall h \in \Gamma_2 \; .
\end{equation}
Make the following assumptions.
\begin{itemize}
\item The restriction of $\al$ to $\Gamma_1$ is ergodic and rigid.
\item We have $\Emb(\Gamma_1,\Gamma_1*\Gamma_2) = [[\Gamma_1*\Gamma_2]]$.
\item $\Gamma_2$ is amenable.
\end{itemize}
Then, the following holds.
\begin{enumerate}
\item\label{one} The map
$$\Theta : \Centr_{\Aut Y}(\Gamma_2) \recht
\Aut(\cR(\Gamma_1 * \Gamma_2 \actson X \times Y)) :
\Delta \mapsto \Theta_\Delta \;\;
$$
$$\text{where}\;\; \Theta_\Delta(x,y) =
(x,\Delta(y))$$
induces an onto group isomorphism between $\Centr_{\Aut Y}(\Gamma_2)$
and $\Out(\cR(\Gamma_1 * \Gamma_2 \actson X \times Y))$.
\item\label{two} Define the II$_\infty$ factor $N := \rL^\infty(X \times Y)
\rtimes (\Gamma_1 * \Gamma_2)$. Denote for every $\Delta \in \Centr_{\Aut Y}(\Gamma_2)$, by $\theta_\Delta$ the corresponding automorphism of $N$.
\begin{enumerate}
\item The group $\Aut(N)$ is generated by the three subgroups $\{\theta_\Delta \mid \Delta \in \Centr_{\Aut Y}(\Gamma_2)\}$, the inner automorphism group $\Inn(N) = \{ \Ad u \mid u \in \cU(N)\}$ and the group of automorphisms\footnote{Note that $H$ is isomorphic to the group of $S^1$-valued $1$-cocycles for the action $\Gamma_1 * \Gamma_2 \actson X \times Y$.} $H := \{\theta \in \Aut(N) \mid \theta(a) = a \;\;\text{for all}\;\; a \in \rL^\infty(X \times Y)\}$.
\item The subgroup $\Inn(N) \cdot H$ of $\Aut(N)$ is closed and normal in $\Aut(N)$ and the map
$$\Centr_{\Aut Y}(\Gamma_2) \recht \frac{\Aut(N)}{\Inn(N) \cdot H} : \Delta \mapsto \theta_\Delta$$
is an isomorphism and homeomorphism of Polish groups.
\end{enumerate}
\end{enumerate}
\end{theorem}
\begin{proof}
The proof of (\ref{one}) is identical to \cite[Lemma 5.1]{PV1}.
It remains to prove (\ref{two}).

Write $A = \rL^\infty(X)$ and $B =
\rL^\infty(Y)$. We first prove that every automorphism of $N$ preserves the Cartan subalgebra $A \ovt B$ up to unitary conjugacy. Together with point 1, this implies 2(a). So, let $\theta$ be an automorphism of $N := (A \ovt B) \rtimes
(\Gamma_1 * \Gamma_2)$. Take a projection $p \in A \ovt B$ of finite trace
and put $q = \theta(p)$. After unitary conjugacy, we may assume that $q \in A \ovt B$.
By \cite[Theorem A.1]{P5}, it is sufficient to prove that $\theta(Ap) \embed{qNq} (A \ovt B)q$.

Since $\theta(Ap) \subset qNq$ is rigid, \cite[Theorem 5.1]{IPP} implies that
$$\theta(Ap) \embed{qNq} q \bigl( (A \ovt B) \rtimes \Gamma_i \bigr)
q \quad\text{for some}\;\; i=1,2.$$
Since $\theta(Ap)$ is quasi-regular in $qNq$, \cite[Theorem 1.1]{IPP}
implies that $\theta(Ap) \embed{qNq} (A \ovt B)q$.

We finally prove 2(b). Observe that $\cU(N)$ is a Polish group in a natural way and that the map $\cU(N) \recht \Aut(N) : u \mapsto \Ad u$ is a continuous group morphism. Define $H$ as in the formulation of the theorem and note that $H$ is a closed subgroup of $\Aut(N)$. We form the semidirect product Polish group $\cU(N) \rtimes H$ in such a way that $\pi : \cU(N) \rtimes H \recht \Aut(N) : \pi(u,\theta) = (\Ad u) \circ \theta$ is a group morphism. Note that $\pi$ is continuous and denote $K := (\cU(N) \rtimes H)/\Ker \pi$. Again, $K$ is a Polish group. We form the semidirect product Polish group $K \rtimes \Centr_{\Aut Y}(\Gamma_2)$ in such a way that
$$\rho : K \rtimes \Centr_{\Aut Y}(\Gamma_2) \recht \Aut(N) : \rho(k,\Delta) = \pi(k) \theta_\Delta$$
is a group morphism. Then $\rho$ is a continuous and injective group morphism between Polish groups. Moreover, by (2a), $\rho$ is onto. So, $\rho$ is a homeomorphism. Hence, $\Inn(N) \cdot H = \rho(K)$ is closed and normal in $\Aut(N)$ and the map $\Delta \mapsto \theta_\Delta$ provides an isomorphism and homeomorphism between $\Centr_{\Aut Y}(\Gamma_2)$ and $\Aut(N) / (\Inn(N) \cdot H)$.
\end{proof}

\begin{lemma} \label{lemma.reduction}
Let $\Gamma * \Lambda \actson (X,\mu)$ be a free p.m.p.\ action with
the restriction to $\Gamma$ being ergodic. Let $\Lambda \actson (Y,\eta)$
be a free ergodic action preserving the infinite standard measure $\eta$.
Assume that $\Lambda$ is amenable. Consider $\Gamma * \Lambda \actson X \times Y$
as in \eqref{eq.action}. Let $Z \subset Y$ be a subset of finite measure and define
the II$_1$ equivalence relation $\cR$ as the restriction of
$\cR(\Gamma * \Lambda \actson X \times Y)$ to $Z$.

Whenever $\Sigma$ is an infinite amenable group, there exists a free
ergodic p.m.p.\ action $\Gamma^{*\infty} * \Sigma \actson X \times Z$
such that $\cR = \cR(\Gamma^{*\infty} * \Sigma \actson X \times Z)$.
\end{lemma}
\begin{proof}
Denote by $\cR_1$ the equivalence relation given by the restriction of
$\cR(\Lambda \actson X \times Y)$ to $X \times Z$. Note that $\cR_1$
need not be ergodic. Since $\Lambda$ is amenable and almost every equivalence
class of $\cR_1$ is infinite, the results in \cite{CFW} and \cite{OW} allow one to
take a free p.m.p.\ action $\Sigma \actson X \times Z$ whose OE
relation is precisely $\cR_1$.

Since the action of $\Lambda$ on $(Y,\eta)$ is ergodic, take $\phi_n \in
[[\Lambda]]$ with $\dom(\phi_n) = X \times Z$ and $\operatorname{range}(\phi_n) =
X \times Z_n$, where $Z_n$, $n \in \N$, forms a partition of $Y$ (up to measure zero).
Since the action of $\Gamma$ leaves every $X \times Z_n$ globally invariant,
we can view $\phi_n^{-1} \Gamma \phi_n$ as a group of automorphisms of $X \times Z$.
It is now an exercise to check that $\cR$ is freely generated by the OE
relations of $\phi_n^{-1} \Gamma \phi_n$, $n \in \N$, together with $\Sigma \actson X \times Z$.
This provides us with the required free action of $\Gamma^{*\infty} * \Sigma \actson X \times Z$.
\end{proof}

The following is the final ingredient in the proof of Theorem \ref{thm.Sfactorlarge}.

\begin{lemma} \label{lemma.two-ergodic}
There exist ergodic measures $\nu,\nu'$ on $\R$ such that
$H_{\nu} \neq \R \neq H_{\nu'}$ and $H_{\nu} + H_{\nu'} = \R$.
\end{lemma}
\begin{proof}
As explained in \cite[Section 2]{PV1}, an ergodic measure $\nu$ on
$\R$ can be associated to any pair $(a_n),(b_n)$ of sequences in $\N$
satisfying $\sum_{n=1}^\infty b_n^{-1} < \infty$ and $b_n < a_n/2$ for
all $n$, in such a way that
$$H_\nu = \Bigl\{ x \in \R \; \Big | \; \sum_{n=1}^\infty \frac{a_n}{b_n}
\| a_1 \cdots a_{n-1} x \| < \infty \Bigr\}$$
where $\|x\|$ denotes the distance of $x \in \R$ to $\Z \subset \R$.
Take $a_n = 2^{2n+2}$, $b_n = 2^{2n}$ and associate with it the ergodic
measure $\nu$. Take $a'_n = 2^{2n+1}$, $b'_n = 2^{2n-1}$ and associate
with it the ergodic measure $\nu'$. First of all,
$$\sum_{n=1}^\infty \frac{b_n}{a_1 \cdots a_n} \not\in H_\nu
\quad\text{and}\quad \sum_{n=1}^\infty \frac{b'_n}{a'_1 \cdots a'_n} \not\in H_{\nu'}$$
proving that $H_\nu \neq \R \neq H_{\nu'}$.

Let now $x \in \R$ and write
$$x = x_0 + \sum_{n=1}^\infty \frac{x_n}{a_1 \cdots a_n}
\quad\text{with}\quad x_n \in \{0,\ldots,a_n - 1\} \; .$$
Write for every $n \in \N$, $x_n = y_n + \sqrt{a_n} z_n$
with $y_n,z_n \in \{0,\ldots,\sqrt{a_n}\}$. Define
$$y = x_0 + \sum_{n=1}^\infty \frac{y_n}{a_1 \cdots a_n}
\quad\text{and}\quad z = \sum_{n=1}^\infty \frac{2z_n}{a'_1 \cdots a'_n} \; .$$
One checks that $y \in H_\nu$, $z \in H_{\nu'}$ and $x = y+z$. So, $H_\nu + H_{\nu'} = \R$.
\end{proof}

\begin{proof}[Proof of Theorem \ref{thm.Sfactorlarge}]\label{proof.thmlarge}
Since $\Gamma^{*\infty} = (\Gamma * \Gamma)^{*\infty}$, we may assume that $\Gamma$ is an infinite group. Let $\Sigma,\Lambda$ be infinite amenable groups and $\Lambda \actson (Y,\eta)$ a free ergodic action
preserving the infinite standard measure $\eta$. Set $G := \Gamma^{*\infty} * \Sigma$. We prove the existence of a free ergodic p.m.p.\ action of $G \actson Z$ such that the associated II$_1$ factor $M := \rL^\infty(Z) \rtimes G$ and equivalence relation $\cR : = \cR(G \actson Z)$ have the following properties.
\begin{enumerate}
\item The fundamental group of $M$ and the fundamental group of $\cR$ equal $\module(\Centr_{\Aut Y}(\Lambda))$.
\item The II$_\infty$ factor $M \ovt \B(\ell^2(\N))$ admits a strongly continuous trace-scaling action of $\Rp$ if and only if the group morphism
   \begin{equation}\label{eq.groupmorphism}
   \module : \Centr_{\Aut Y}(\Lambda) \recht \Rp
   \end{equation}
   is onto and splits continuously.
\end{enumerate}

Choose any free p.m.p.\ actions $\Gamma \actson (X,\mu)$ and $\Lambda \actson (X,\mu)$.
Take a free p.m.p.\ action $\Gamma^{*\infty} * \Lambda \actson (X,\mu)$ satisfying the
conclusions of Lemma \ref{lemma.nosymmetry}. Define $\Gamma^{*\infty} * \Lambda \actson
X \times Y$ by \eqref{eq.action}, with $\Gamma_1 = \Gamma^{*\infty}$ and $\Gamma_2 = \Lambda$.
Define the II$_1$ equivalence relation $\cR$ by restricting
$\cR(\Gamma^{*\infty} * \Lambda \actson X \times Y)$ to a subset $Z$
of finite measure. By Lemma \ref{lemma.reduction}, we can take a free ergodic
p.m.p.\ action $G \actson Z$ whose OE relation equals $\cR$. By (1) of Theorem \ref{thm.compute-out},
$$\cF(\cR) = \module(\Centr_{\Aut Y}(\Lambda)) \; .$$
Put $M := \rL^\infty(Z) \rtimes G$ and note that $M \ovt \B(\ell^2(\N)) \cong \rL^\infty(X \times Y) \rtimes (\Gamma^{*\infty} * \Lambda)$.
By (2a) of Theorem \ref{thm.compute-out}, also $$\cF(M) = \module(\Aut(N)) = \module(\Centr_{\Aut Y}(\Lambda)) \; .$$
If the group morphism \eqref{eq.groupmorphism} splits continuously, it is clear that $N$ admits a strongly continuous trace scaling action. The converse follows from (2b) of Theorem \ref{thm.compute-out}.

In order to conclude the proof of Theorem \ref{thm.Sfactorlarge}, we have to construct an action $\Lambda \actson (Y,\eta)$ such that $\module(\Centr_{\Aut Y}(\Lambda)) = \Rp$, but the morphism \eqref{eq.groupmorphism} does not split continuously.
By Lemma \ref{lemma.two-ergodic}, we can take ergodic measures
$\nu_1,\nu_2$ on $\R$ such that $H_{\nu_1} \neq \R \neq H_{\nu_2}$,
while $H_{\nu_1}+H_{\nu_2} = \R$. By formula (2.2) in \cite{PV1}, we
can take amenable groups $\Lambda_1,\Lambda_2$ and free ergodic infinite
measure preserving actions $\Lambda_i \actson (Y_i,\eta_i)$ such that
$\module(\Centr_{\Aut Y_i}(\Lambda_i)) = \exp(H_{\nu_i})$. Since the
homomorphism $\module$ is continuous, we equip $\exp(H_{\nu_i})$ with
the (Polish) quotient topology. In this way, the $H_{\nu_i}$ become
Polish groups and the embedding $H_{\nu_i} \hookrightarrow \R$ continuous.

We prove that
$$\module : \Centr_{\Aut( Y_1 \times Y_2)}(\Lambda_1 \times \Lambda_2) \recht \R_+$$
admits no continuous splitting. Assume that it does. Since the left-hand
side of the previous formula equals $\Centr_{\Aut Y_1}(\Lambda_1) \times
\Centr_{\Aut Y_2}(\Lambda_2)$, the homomorphism
$$H_{\nu_1} \times H_{\nu_2} \recht \R : (x,y) \mapsto x+y$$
admits a continuous splitting. We then find continuous homomorphisms
$\theta_i : \R \recht H_{\nu_i}$ such that $x = \theta_1(x) + \theta_2(x)$
for all $x \in \R$. Since the embedding $H_{\nu_i} \hookrightarrow \R$ is
continuous, there exist $\lambda_i \in \R$ such that $\theta_i(x) =
\lambda_i x$ for all $x \in \R$. But, $H_{\nu_i} \neq \R$, forcing
$\lambda_i = 0$ for $i=1,2$, a contradiction.
\end{proof}

\begin{proof}[Proof of Theorem \ref{thm.tensor-product}]
By Lemma \ref{lemma.two-ergodic}, we can take ergodic measures
$\nu_1,\nu_2$ on $\R$ such that $H_{\nu_1} \neq \R \neq H_{\nu_2}$,
while $H_{\nu_1}+H_{\nu_2} = \R$. By formula (2.2) in \cite{PV1},
$\exp(H_{\nu_i}) \in \Scentr$, so that by Theorem \ref{thm.Sfactorlarge},
we can take II$_1$ factors $M_i$ with $\cF(M_i) = \exp(H_{\nu_i})$.
Then the fundamental group of $M_1 \ovt M_2$ contains $\exp(H_{\nu_1} + H_{\nu_2})$
and hence equals $\R_+$.
\end{proof}

\section{Groups $G$ for which $\Sfactor(G)$ is trivial} \label{sec.Sfactortrivial}

Combining results from \cite{CH,G1,IPP}, we prove that for the following
groups $G$, $\Sfactor(G)$ is trivial.

\begin{theorem} \label{thm.Sfactortrivial}
Let $\Gamma$ and $\Lambda$ be infinite, finitely generated groups.
Assume that $\Gamma$ is ICC and that one of the following conditions holds.
\begin{enumerate} \renewcommand{\theenumi}{\alph{enumi}}\renewcommand{\labelenumi}{\theenumi)}
\item\label{a} $\Gamma = \Gamma_1 \times \Gamma_2$ is a non-trivial direct product
with $\Gamma_2$ being non-amenable,
\item\label{b} $\Gamma$ admits a non-virtually abelian, normal subgroup $\Gamma_1$
with the relative property (T).
\end{enumerate}
Then $\Sfactor(\Gamma * \Lambda) = \Seqrel(\Gamma*\Lambda) =
\{\{1\}\}$.
\end{theorem}

\begin{remark} \label{rem.smallvslarge}
Observe that Theorem \ref{thm.Sfactortrivial} implies that, in general,
Theorem \ref{thm.Sfactorlarge} is false if we only take a finite free product.
\end{remark}

\begin{proof}[Proof of Theorem \ref{thm.Sfactortrivial}]
Let $\Gamma * \Lambda \actson (X,\mu)$ be a free ergodic p.m.p.\ action.
Note first that by \cite[Propri\'{e}t\'{e}s 1.5]{G1}, we have
$0 < \beta_1^{(2)}(\Gamma * \Lambda) < \infty$. Hence, by \cite[Corollaire 5.7]{G1},
the fundamental group of the OE relation $\cR(\Gamma * \Lambda \actson X)$
is trivial.

Write $A = \rL^\infty(X)$, $M_1 = A \rtimes \Gamma$, $M_2 = A \rtimes \Lambda$.
Finally, set $M = A \rtimes (\Gamma * \Lambda) = M_1 *_A M_2$. Suppose that $p \in A$
is a projection and $\theta : M \recht pMp$ a $^*$-isomorphism. It remains to prove
that $\theta(A)$ and $pA$ are unitarily conjugate, since this implies that
$\cF(M) = \cF(\cR(\Gamma * \Lambda \actson X))$.

Under assumption \ref{a}), we invoke \cite[Theorem 4.2]{CH} and under
assumption \ref{b}), we invoke \cite[Theorem 5.1]{IPP} and conclude
in both cases that $\theta(L(\Gamma_1)) \embed{M} M_i$ for some $i=1,2$.
Take a projection $q \in M_i^n$, a non-zero partial isometry $v \in p(\M_{1,n}(\C) \ot M)q$
and a unital $^*$-homomorphism $\rho : L(\Gamma_1) \recht q M_i^n q$ satisfying
$\theta(a) v = v \rho(a)$ for all $a \in L(\Gamma_1)$. In both cases \ref{a})
and \ref{b}), the group $\Gamma_1$ is not virtually abelian. Hence,
$\rho(L(\Gamma_1)) \notembed{M_i} A$. By \cite[Theorem 1.1]{IPP},
the normalizer of $\rho(L(\Gamma_1))$ inside $qM^n q$ is contained
in $q M_i^n q$. Since $v^* v$ commutes with $\rho(L(\Gamma_1))$,
we may first of all assume that $q = v^* v$. Next, it follows that
$v^* \theta(L(\Gamma)) v \subset q M_i^n q$. Hence, $\theta(L(\Gamma)) \embed{M} M_i$.

Repeating the previous paragraph, we may assume that $\rho : L(\Gamma) \recht q M_i^n q$,
$\theta(a) v = v \rho(a)$ for all $a \in L(\Gamma)$ and $v^* v = q$. Since $\Gamma$
is an ICC group, we get that
$$M \cap L(\Gamma)' = M_1 \cap L(\Gamma)' = A^\Gamma \; .$$
So, $vv^* \in \theta(A^\Gamma)$. It follows that $v^* \theta(A) v$ is a
Cartan subalgebra of $q M^n q$. Moreover, for all $g \in \Gamma$,
the unitary $v^* \theta(u_g) v = \rho(u_g)$ belongs to $q M_i^n q$
and normalizes $v^* \theta(A) v$. Then, \cite[Theorem 1.8]{IPP}
implies that there exists $w \in q M^n$ such that $ww^* = q$,
$w^* w \in A^n$ and $w^* v^* \theta(A) v w = w^* w A^n$.
It follows that $\theta(A)$ and $pA$ are unitarily conjugate.
\end{proof}

\section{$\Sfactor(\Z^n \rtimes \SL(n,\Z))$ is non-trivial, for all $ n\geq 3$}
\label{sec.Trational}

When $\Gamma$ is an ICC property (T) group, all groups in
$\Sfactor(\Gamma)$ are countable (cf. \cite[Proof of Theorem
1.7]{gg}, or \cite[Theorem 4.5.1]{Pcorr}). Nevertheless,
$\Sfactor(\Gamma)$ can be non-trivial, as shown by the next theorem,
in which we show that if $\Gamma=\Z^n \rtimes \SL(n,\Z)$,
$n\geq 3$,  then $\Sfactor(\Gamma)$ contains ``many'' subgroups of
$\Q_+$. It is unclear though whether there exists a free ergodic
p.m.p.\ action $\Gamma \actson (X,\mu)$ of an ICC property (T) group
$\Gamma$ such that $\cF(\cR(\Gamma \actson X)) \not\subset \Q_+$.

\begin{theorem} \label{thm.rational}
Let $\cF \subset \Q_+$ be a subgroup generated by a subset of the
prime numbers. Let $\Gamma = \Z^n \rtimes \SL(n,\Z)$ with $n \geq
3$. Then $\Gamma$ admits a ``concrete'' free ergodic p.m.p.\ action
$\Gamma \actson (X,\mu)$ such that both the fundamental group of
$\rL^\infty(X) \rtimes \Gamma$ and of $\cR(\Gamma \actson X)$ equal
$\cF$.
\end{theorem}

We prove Theorem \ref{thm.rational} as a consequence of the following more general result.

\begin{theorem}\label{thm.Tdiagonal}
Let $\Gamma$ be a group having a normal, non-virtually abelian subgroup $\Sigma$
with the relative property (T) and with $\Gamma / \Sigma$ being finitely generated.
Let $\Gamma \supset \Gamma_1 \supset \Gamma_2 \supset \cdots$ be a decreasing
sequence of finite index subgroups such that the action $\Gamma \actson
(X,\mu) := \invlimit \Gamma / \Gamma_n$ is essentially free. Consider
the diagonal product action $\Gamma \actson X \times [0,1]^\Gamma$ of
$\Gamma \actson X$ and the Bernoulli action $\Gamma \curvearrowright  [0,1]^\Gamma$.

Then the fundamental groups of the associated II$_1$ factor and
II$_1$ equivalence relation are both equal to
\begin{equation}\label{eq.candidate}
\begin{split}
\Bigl\{ \frac{[\Gamma : \Lambda_1]}{[\Gamma : \Lambda_2]} \; \Big | \; &
\Gamma_{n} \subset \Lambda_1 \cap \Lambda_2 \;\;\text{for large enough
$n$, and}\\ & \Lambda_1 \actson \invlimit \Lambda_1 /
\Gamma_n \;\;\text{is conjugate with}\;\; \Lambda_2 \actson \invlimit \Lambda_2 /
\Gamma_n \; \Bigr\} \; .
\end{split}
\end{equation}
\end{theorem}

Conjugacy of two profinite actions can be expressed in purely group-theoretic
terms; see e.g.\ \cite[Proposition 1.8]{ioana}.

Before proving Theorem \ref{thm.Tdiagonal}, we introduce some terminology
and an auxiliary result. Recall that a \emph{$1$-cocycle}
$\om : \Gamma \times X \recht \Lambda$ for an action $\Gamma \actson (X,\mu)$
with values in a countable group $\Lambda$ is a measurable map satisfying
$$\om(gh,x) = \om(g,h \cdot x) \, \om(h,x) \quad\text{for all}\;\; g,h \in
\Gamma \;\;\text{and almost all}\;\; x \in X \; .$$
The $1$-cocycles $\om,\om' : \Gamma \times X \recht \Lambda$ are called
\emph{cohomologous} if there exists a measurable map $\vphi : X \recht \Lambda$
satisfying $\om'(g,x) = \vphi(g \cdot x) \om(g,x) \vphi(x)^{-1}$ almost everywhere.
We identify homomorphisms from $\Gamma$ to $\Lambda$ with $1$-cocyles $\om$ that
are independent of the $x$-variable.

\begin{definition}\label{def.virtual-cocycle}
Let $\Gamma \actson (X,\mu)$ be a free ergodic p.m.p.\ action. We say that a
$1$-cocycle $\om : \Gamma \times X \recht \Lambda$ \emph{virtually untwists} if there exists
\begin{itemize}
\item a finite index subgroup $\Gamma_0 < \Gamma$ and a measurable map
$\pi : X \recht \Gamma/\Gamma_0$ satisfying $\pi(g \cdot x) = g \pi(x)$
almost everywhere,
\item a $1$-cocycle $\om' : \Gamma \times \Gamma/\Gamma_0 \recht \Lambda$
for the action $\Gamma \actson \Gamma / \Gamma_0$,
\end{itemize}
such that $\om$ is cohomologous to the $1$-cocycle $(g,x) \mapsto \om'(g,\pi(x))$.

We call $\Gamma \actson (X,\mu)$ \emph{virtually cocycle superrigid}
(with countable target groups) if every $1$-cocycle with values in a
countable group $\Lambda$ virtually untwists.
\end{definition}

A \emph{stable orbit equivalence} between free ergodic p.m.p.\
actions $\Gamma \actson (X,\mu)$ and $\Lambda \actson (X',\mu')$ is
a map $\Delta : X \recht X'$ satisfying the following properties.
\begin{itemize}
\item For almost every $x \in X$, we have $\Delta(\Gamma \cdot x) =
\Lambda \cdot \Delta(x)$.
\item There exists a partition $X = \bigsqcup_n X_n$ of $X$ into
measurable subsets $X_n \subset X$ and there exist measurable
subsets $X'_n \subset X'$ such that for every $n \in \N$, the
restriction of $\Delta$ to $X_n$ is a non-singular isomorphism
between $X_n$ and $X'_n$.
\end{itemize}
By ergodicity, all of these non-singular isomorphisms
$\Delta|_{X_n} : X_n \recht X'_n$ are measure scaling, with the
scaling being independent of $n$. This scaling factor is called the
\emph{compression constant} of the stable OE $\Delta$
and denoted by $c(\Delta)$.

The \emph{Zimmer $1$-cocycle} $\om : \Gamma \times X \recht \Lambda$ associated
with the stable OE $\Delta$ is defined by
$$\Delta(g \cdot x) = \om(g,x) \cdot \Delta(x) \quad\text{almost everywhere}.$$
Two stable OEs $\Delta_1,\Delta_2 : X \recht X'$ are
called \emph{similar} if $\Delta_1(x) \in \Lambda \cdot \Delta_2(x)$
for almost all $x \in X$. Note that similar stable OEs give rise to cohomologous $1$-cocycles.

Whenever $X_0 \subset X$ and $X'_0 \subset X'$ are non-negligible
measurable subsets and $\Delta_0 : X_0 \recht X'_0$ is a
non-singular isomorphism satisfying $\Delta_0(\Gamma \cdot x \cap
X_0) = \Lambda \cdot \Delta_0(x) \cap X'_0$ for almost all $x \in
X_0$, ergodicity allows one to choose a measurable map $p : X \recht
X_0$ with $p(x) \in \Gamma \cdot x$ for almost all $x \in X$ and
then, $\Delta := \Delta_0 \circ p$ defines a stable OE. Another choice of $p$
gives rise to a similar stable
OE. It follows that
$$\cF(\cR(\Gamma \actson X)) = \{ c(\Delta) \mid \Delta \;\;\text{is a stable
OE between}\;\; \Gamma \actson X \;\;\text{and}\;\; \Gamma \actson X \} \; .$$

Let $\Gamma \actson (X,\mu)$. We say that the action $\Gamma \actson X$
is \emph{induced} from $\Gamma_1 \actson X_1$ if $X_1$ is a non-neglible
measurable subset of $X$ and $\Gamma_1 < \Gamma$ is a finite index subgroup
such that $g \cdot X_1 = X_1$ for all $g \in \Gamma_1$ and
$\mu(g \cdot X_1 \cap X_1) = 0$ if $g \in \Gamma - \Gamma_1$.
Obviously, in this situation $\Gamma \actson X$ is stably orbit
equivalent to $\Gamma_1 \actson X_1$ with compression constant $[\Gamma : \Gamma_1]^{-1}$.

The following provides one more instance of a general principle
going back to \cite[Proposition 4.2.11]{Z2}. For other versions
of this, see \cite[Proposition 5.11]{P3} and  \cite[Lemma 4.7]{V1}.

\begin{proposition} \label{prop.oe}
Let $\Delta : X \recht X'$ be a stable OE between the
free ergodic p.m.p.\ actions $\Gamma \actson (X,\mu)$ and $\Lambda
\actson (X',\mu')$. If the associated Zimmer $1$-cocycle virtually
untwists (see Definition \ref{def.virtual-cocycle}), there exist
finite index subgroups $\Gamma_1 < \Gamma$, $\Lambda_1 < \Lambda$,
non-negligible measurable subsets $X_1 \subset X$, $X'_1 \subset X'$
and a finite normal subgroup $H \lhd \Gamma_1$ such that
\begin{enumerate}
\item $\Gamma \actson X$ is induced from $\Gamma_1 \actson X_1$,
\item $\Lambda \actson Y$ is induced from $\Lambda_1 \actson Y_1$,
\item the actions $\Gamma_1 / H \actson X_1/H$ and $\Lambda_1 \actson Y_1$ are conjugate,
\end{enumerate}
and such that the stable OE $\Delta$ is similar to the
composition of the canonical stable OEs given by (1), (3) and (2).
In particular, the compression constant of $\Delta$ equals
$$c(\Delta) = \frac{[\Lambda : \Lambda_1]}{[\Gamma : \Gamma_1] |H|} \; .$$
\end{proposition}
\begin{proof}
Let $\Delta(g \cdot x) = \om(g,x) \cdot \Delta(x)$ almost everywhere.
By our assumption, take a finite index subgroup $\Gamma_1$, a quotient map
$\pi : X \recht \Gamma/\Gamma_1$ and a $1$-cocycle $\om' :
\Gamma \times \Gamma/\Gamma_1 \recht \Lambda$ such that $\pi(g \cdot x) = g \pi(x)$
almost everywhere and such that $\om$ is cohomologous to the $1$-cocycle
$(g,x) \mapsto \om'(g,\pi(x))$. Define $X_1 = \pi^{-1}(e \Gamma_1)$.
By construction, $\Gamma \actson X$ is induced from $\Gamma_1 \actson X_1$.

Denote by $\Delta_1$ the restriction of $\Delta$ to $X_1$. Then $\Delta_1$
is a stable OE between $\Gamma_1 \actson X_1$ and $\Lambda \actson Y$.
By construction, the $1$-cocycle associated with $\Delta_1$ is cohomologous to a
homomorphism from $\Gamma_1$ to $\Lambda$. The conclusion of the proposition now
follows from \cite[Lemma 4.7]{V1}.
\end{proof}

In order to show the equality of the fundamental groups of the II$_1$ factor
and the II$_1$ equivalence relation associated with the $\Gamma$-actions
defined in Theorem \ref{thm.Tdiagonal}, we need the following result about
automatic preservation of Cartan subalgebras.

\begin{proposition} \label{prop.Cartanpreserving}
Let $\Gamma$ be a countable group having a normal, non-virtually abelian
subgroup $\Sigma$ with the relative property (T). Let $\Gamma \actson (X,\mu)$
be a free ergodic p.m.p.\ action and assume that this action admits a free and
profinite quotient: there exists a free profinite p.m.p.\ action
$\Gamma \actson (X_1,\mu_1)$ and a quotient map $\pi : X \recht X_1$
satisfying $\pi(g \cdot x) = g \cdot \pi(x)$ almost everywhere.

Let $(Y_0,\eta_0)$ be a non-trivial standard probability space and set
$(Y,\eta) = (Y_0,\eta_0)^\Gamma$. Consider the diagonal action
$\Gamma \actson (X \times Y,\mu \times \eta)$. Set $M = \rL^\infty(X \times Y) \rtimes \Gamma$.

Then every isomorphism $\theta : M \recht pMp$ preserves, up to unitary conjugacy,
the natural Cartan subalgebras of $M$, $pMp$.
\end{proposition}

\begin{proof}
Set $A = \rL^\infty(X)$ and $B = \rL^\infty(Y)$. Let $\theta : M \recht pMp$
be an isomorphism. Denote $A_1 = \rL^\infty(X_1)$ and view $A_1$ as a globally
$\Gamma$-invariant von Neumann subalgebra of $A$.

Almost literally repeating \cite[Theorem 4.1]{P1} (see also \cite[Lemma 6.1]{V1}),
we find that $\theta(L(\Sigma)) \embed{M} A \rtimes \Gamma$. Take a projection
$q \in (A \rtimes \Gamma)^n$, a non-zero partial isometry $v \in p (\M_{1,n}(\C) \ot M)q$
and a unital $^*$-homomorphism $\al : L(\Sigma) \recht q(A \rtimes \Gamma)^n q$ such
that $\theta(a) v = v \al(a)$ for all $a \in L(\Sigma)$.

Since $\Sigma$ is normal in $\Gamma$ and since $\Sigma \actson X_1$ is profinite,
the quasi-normalizer of $L(\Sigma)$ inside $M$ contains $A_1 \rtimes \Gamma$.
Since $\Sigma$ is non-virtually abelian, $L(\Sigma)$ cannot be embedded in an
amplification of $A$. So, by \cite[Proposition D.5]{V1},
$$v^* \theta(A_1 \rtimes \Gamma) v \subset (A \rtimes \Gamma)^n \; .$$
It follows in particular that $\theta(A_1) \embed{M} A \rtimes \Gamma$.
We claim that in fact $\theta(A_1) \embed{M} A$. Indeed, if this were not the case, applying once more \cite[Proposition D.5]{V1}
(and using the regularity of $A_1 \subset (A \ovt B) \rtimes \Gamma$)
would yield $M \embed{M} A \rtimes \Gamma$, a contradiction. This proves the claim.

Since $\Gamma \actson X_1$ is free, we have $A_1' \cap M = A \ovt B$.
Hence, the proposition follows from \cite[Theorem A.1]{P5}.
\end{proof}

We are now ready to prove Theorem \ref{thm.Tdiagonal}.

\begin{proof}[Proof of Theorem \ref{thm.Tdiagonal}]
Put $(X,\mu) = \invlimit \Gamma/\Gamma_n$ as in the formulation of the theorem.
We assume $\Gamma \actson X$ to be essentially free.
Let $Y = [0,1]^\Gamma$ and denote by $\eta$ the infinite product of the Lebesgue
measure on $[0,1]$. We consider the diagonal action $\Gamma \actson X \times Y$.

By Proposition \ref{prop.Cartanpreserving}, we have
$$\cF(\rL^\infty(X \times Y) \rtimes \Gamma) = \cF(\cR(\Gamma \actson X \times Y)) \; .$$
Whenever $\Lambda_1 < \Gamma$ is a subgroup containing $\Gamma_n$ for large enough $n$,
the action $\Gamma \actson (X,\mu)$ is induced from the action $\Lambda_1 \actson X_1 :=
\invlimit \Lambda_1 / \Gamma_n$, and hence $\Gamma \actson X \times Y$ is induced from
$\Lambda_1 \actson X_1 \times Y$. Since $\Lambda_1 \acts [0,1]^\Gamma$ and
$\Lambda_1 \acts [0,1]^{\Lambda_1}$ are isomorphic actions, it follows that
the set defined by \eqref{eq.candidate} is part of the fundamental group
$\cF(\cR(\Gamma \actson X \times Y))$.

A combination of \cite[Theorem 0.1]{P3} and \cite[Theorem B]{ioana} yields
that the diagonal action $\Gamma \actson X \times Y$ is virtually cocycle
superrigid in the sense of Definition \ref{def.virtual-cocycle}. So, we can
apply Proposition \ref{prop.oe}.

Let $\Delta : X \times Y \recht X \times Y$ be a stable OE
between $\Gamma \actson X \times Y$ and itself. We have to prove that
$c(\Delta)$ belongs to the set defined in \eqref{eq.candidate}.
Proposition \ref{prop.oe} provides us with finite index subgroups
$G_1,G_2$ of $\Gamma$, non-negligible measurable subsets $Z_1,Z_2 \subset X \times Y$
and a finite normal subgroup $H$ of $G_1$ such that $\Gamma \actson X \times Y$
is induced from $G_i \actson Z_i$ and such that $G_1/H \actson Z_1/H$ is conjugate
to $G_2 \actson Z_2$, say through the isomorphism $\Delta : Z_1/H \recht Z_2$ and
the group isomorphism $\delta : G_1 \recht G_2$.
Finally, $c(\Delta) = \frac{[\Gamma : G_2]}{[\Gamma : G_1] |H|}$.

Since the Bernoulli action $G_i \actson Y$ is mixing, we have
$Z_i = X_i \times Y$, with $\Gamma \actson X$ being induced from
$G_i \actson X_i$. Moreover, still because the Bernoulli action
is mixing, $\Delta(x,y) = (\Delta_0(x),\ldots)$, where $\Delta_0 : X_1/H \recht X_2$
is an isomorphism conjugating the actions $G_1/H \actson X_1/H$ and $G_2 \actson X_2$
through the group isomorphism $\delta : G_1/H \recht G_2$.

Denote by $\pi_n : X \recht \Gamma/\Gamma_n$ the natural quotient map.
By \cite[Lemma 4.1]{ioana}, we find $k \in \N$ and $g \in \Gamma$ such
that $g \Gamma_k g^{-1} \subset G_1$ and $X_1 = \pi_k^{-1}(G_1 g \Gamma_k)$.
Moreover, since $\Gamma \actson X$ is free, we can take $k$ large enough and
assume that $H \cap \Gamma_k = \{e\}$. Replacing $G_1$ by $g^{-1} G_1 g$ and
$X_1$ by $g^{-1} \cdot X_1$, we may assume that $g = e$. Note that
$G_1/H \actson X_1/H = \invlimit G_1 / (\Gamma_n H)$ is induced from
$(\Gamma_k H)/H \actson \invlimit (\Gamma_k H) / (\Gamma_n H)$ and that the latter is
conjugate to $\Gamma_k \actson \invlimit \Gamma_k/\Gamma_n$, because $\Gamma_k \cap H = \{e\}$.

It follows that $\Gamma \actson X$ is induced from $\delta((\Gamma_k H)/H)
\actson \Delta_0(\invlimit (\Gamma_k H) / (\Gamma_n H))$. Applying as above
\cite[Lemma 4.11]{ioana}, we find $h \in \Gamma$ such that, after replacing
$\delta$ by $g \mapsto h \delta(g) h^{-1}$ and $\Delta_0$ by $x \mapsto h \cdot \Delta_0(x)$,
we have $\Gamma_n \subset \Lambda_2 := \delta((\Gamma_k H)/H)$ for $n$ large enough and
$$\Delta_0(\invlimit (\Gamma_k H) / (\Gamma_n H)) = \invlimit \Lambda_2 / \Gamma_n \; .$$
Denoting $\Lambda_1 = \Gamma_k$, we have constructed finite index subgroups
$\Lambda_1,\Lambda_2 \subset \Gamma$ such that $\Gamma_n \subset \Lambda_1 \cap \Lambda_2$
for $n$ large enough and such that the actions $\Lambda_i \actson \invlimit \Lambda_i / \Gamma_n$
are conjugate for $i=1,2$. Tracing back the construction, we also have
$$c(\Delta) = \frac{[\Gamma : \Lambda_2]}{[\Gamma : \Lambda_1]}$$
concluding the proof of the theorem.
\end{proof}

\begin{proof}[Proof of Theorem \ref{thm.rational}]
Let $\cF$ be a subgroup of $\Q_+$ generated by a non-empty subset $\cP$ of the prime numbers.
The case $\cF = \{1\}$ will be discussed at the end of the proof. Denote by $R$ the subring
of $\Q$ generated by $\cP^{-1}$. Note that $R^* = \cF \cup (- \cF)$. Set $G = R^n \rtimes
\GL(n,R)$ and $\Gamma = \Z^n \rtimes \SL(n,\Z)$. Let $G = \{g_1,g_2,\dots\}$ and define
the finite index subgroups $\Gamma_k < \Gamma$ as $\Gamma_k = \Gamma \cap \bigcap_{i=1}^k g_i
\Gamma g_i^{-1}$. Define the profinite action $\Gamma \actson (X,\mu) := \invlimit \Gamma /
\Gamma_k$.

We first argue why $\Gamma \actson (X,\mu)$ is essentially free. Let $p \in \cP$ and take
$k_1 < k_2 < k_3 < \cdots$ such that $(0,p^l 1) \in \{g_1,\ldots,g_{k_l}\}$. One checks
that $\Gamma_{k_l} \subset G_l := p^l \Z^n \rtimes \SL(n,\Z)$, and hence it suffices to
prove freeness of $\Gamma \actson \invlimit \Gamma / G_l$. The latter has been shown in
\cite[discussion before Corollary 5.8]{ioana}.

Consider the diagonal action $\Gamma \actson X \times [0,1]^\Gamma$. Denote by $\cF$ the
set defined in \eqref{eq.candidate}.
By Theorem \ref{thm.Tdiagonal}, we have to prove that $\cF = R^*_+$. It is more convenient
to write $X = \invlimit \Gamma/\Gamma_F$, where $F$ runs through the finite subsets of $G$
and $\Gamma_F := \Gamma \cap \bigcap_{g \in F} g \Gamma g^{-1}$. Whenever $g \in G$, the
action $\Gamma \actson X$ is induced from
$$\Gamma \cap g \Gamma g^{-1} \actson \invlimit_{g \in F \subset G} \frac{\Gamma \cap g
\Gamma g^{-1}}{\Gamma_F}$$
and is induced from
$$g^{-1} \Gamma g \cap \Gamma \actson \invlimit_{g^{-1} \in F \subset G} \frac{g^{-1} \Gamma
g \cap \Gamma}{\Gamma_F} \; .$$
The two actions are conjugate by construction. If $g = (x,A)$, one checks that
$$\frac{[\Gamma : \Gamma \cap g \Gamma g^{-1}]}{[\Gamma : g^{-1} \Gamma g \cap \Gamma]}
= |\det A| \; .$$
It follows that $R^*_+ \subset \cF$.

Conversely, we claim that whenever $\Lambda_1 , \Lambda_2 < \Gamma$ are isomorphic finite
index subgroups of $\Gamma$ such that $\Gamma_k \subset \Lambda_1 \cap \Lambda_2$ for some
$k$, then $[\Gamma : \Lambda_1]/[\Gamma : \Lambda_2] \in R^*_+$. Once this claim is proven,
we get the required equality $\cF = R^*_+$. Let $\delta : \Lambda_1 \recht \Lambda_2$ be an
isomorphism and $\Gamma_k \subset \Lambda_1 \cap \Lambda_2$. We have
$\delta(\Lambda_1 \cap \Z^n) = \Lambda_2 \cap \Z^n$. An elementary argument
for this fact can be given by repeating the beginning of the proof of
\cite[Proposition 7.1]{PV2}. Denoting by $\pi : \Gamma \recht \SL(n,\Z)$
the quotient map, $\pi(\Lambda_1)$ and $\pi(\Lambda_2)$ are isomorphic
finite index subgroups of $\SL(n,\Z)$. Using \cite[Lemma 5.2]{ioana}, it
follows that $[\SL(n,\Z) : \pi(\Lambda_1)] = [\SL(n,\Z) : \pi(\Lambda_2)]$.
Hence, we get
$$\frac{[\Gamma : \Lambda_1]}{[\Gamma : \Lambda_2]} = \frac{[\SL(n,\Z) :
\pi(\Lambda_1)] \; [\Z^n : \Z^n \cap \Lambda_1]}{[\SL(n,\Z) : \pi(\Lambda_2)]
\; [\Z^n : \Z^n \cap \Lambda_2]} = \frac{[\Z^n : \Z^n \cap \Lambda_1]}
{[\Z^n : \Z^n \cap \Lambda_2]} \; .$$
Being finite index subgroups of $\Z^n$, we have $\Lambda_i \cap \Z^n = B_i \Z^n$
for some $B_i \in \M_n(\Z)$ with $\det B_i \neq 0$, $i=1,2$. It follows that
there exists $A \in \GL(n,\Q)$ such that $\delta(x,1) = (A x, 1)$ for all
$(x,1) \in \Lambda_1 \cap \Z^n$. Hence,
$$\frac{[\Z^n : \Z^n \cap \Lambda_1]}{[\Z^n : \Z^n \cap \Lambda_2]} =
\frac{[\Z^n : \Z^n \cap A^{-1} \Z^n]}{[\Z^n : \Z^n \cap A \Z^n]} \; .$$
Since
$$\Z^n \cap \Gamma_k \subset \Z^n \cap \Lambda_1 \cap \Lambda_2 \subset \Z^n \cap A \Z^n
\cap A^{-1} \Z^n \; ,$$ we find $\al \in R^* \cap (\N - \{0\})$ such that
$\al \Z^n \subset \Z^n \cap A \Z^n \cap A^{-1} \Z^n$ for $i=1,2$.
We conclude that $A \in \GL(n,R)$ and finally,
$$\frac{[\Gamma : \Lambda_1]}{[\Gamma : \Lambda_2]} = |\det A|^{-1} \in R^*_+ \; .$$

To conclude the proof of the theorem, we need to construct a free
ergodic p.m.p.\ action $\Gamma \actson (X,\mu)$ such that the
associated II$_1$ factor has trivial fundamental group. By
\cite[Corollary 0.2]{P2}, the Bernoulli action $\Gamma \acts
[0,1]^\Gamma$ has this property. Other examples can be given as
follows. Let $p_1,p_2,\ldots$ be an
enumeration of the prime numbers
and set $\Gamma_k = p_1 \cdots p_k \Z^n \rtimes \SL(n,\Z)$. By
Theorem \ref{thm.Tdiagonal} and \cite[Corollary 5.8]{ioana}, the
diagonal product of the Bernoulli action $\Gamma \acts [0,1]^\Gamma$
and the profinite action $\Gamma \actson \invlimit \Gamma/\Gamma_k$,
provides a crossed product II$_1$ factor with trivial fundamental
group.
\end{proof}

\section{Property (T) and countability of the fundamental group} \label{sec.countable}

In his celebrated ``rigidity paper'' \cite{connes}, Connes showed
that II$_1$ factors arising from ICC groups with the property (T) of
Kazhdan have countable fundamental group. Using the same ideas, it
was later shown that, for a separable II$_1$ factor $M$ to have
countable $\cF(M)$, it is in fact sufficient that $M$ contains a
subfactor with the property (T) in the sense of \cite{CJ2} and
having trivial relative commutant, $N'\cap M=\C$ (cf.\ Theorem 4.6.1
in \cite{Pcorr}; see also \cite{NPS} for a more general statement).
It was also shown that if $M$ is a separable II$_1$ factor, then the
family of subfactors $N_i\subset M, i\in I$, having property (T) and
trivial relative commutant, is countable modulo conjugacy by
unitaries in $M$ (cf.\ Theorem 4.5.1 in \cite{Pcorr}; see also
\cite{ozawa} for a related result). In this section, we prove some
analogous  results for II$_1$ equivalence relations.

In particular, these results show that given any Kazhdan group
$\Gamma$, $\Seqrel(\Gamma)$ can only contain countable subgroups of
$\R_+$, and if in addition $\Gamma$ is ICC then the same holds true
for $\Sfactor(\Gamma)$. More generally, Part 1 of Theorem
\ref{thm.TMcDuff} below shows that this is still the case if the
center of $\Gamma$ ``virtually coincides'' with its {\it FC radical}
(as defined before \ref{thm.TMcDuff}). However, if one drops this
assumption on $\Gamma$, then the situation becomes quite
complicated. Thus, Part 2 of Theorem \ref{thm.TMcDuff} shows that if
a property (T) group $\Gamma$ is residually finite and has
non-virtually abelian FC radical, then $\Gamma$ admits free ergodic
p.m.p.\ actions $\Gamma \actson (X,\mu)$ such that $\rL^\infty(X)
\rtimes \Gamma$ is McDuff and hence its fundamental group is equal
to $\R_+$.

We first need some notation. Thus, if $\cR$ is a II$_1$
equivalence relation on the standard probability space $(X,\mu)$,
then we denote by $[\cR]$ the \emph{full group} of the equivalence
relation $\cR$, consisting of all non-singular isomorphisms $\phi :
X \recht X$ satisfying $(x,\phi(x)) \in \cR$ for almost all $x \in
X$. The \emph{full pseudogroup} of $\cR$ is denoted by $[[\cR]]$ and
consists of all non-singular partial automorphisms $\phi$ between
measurable subsets $D(\phi), R(\phi) \subset X$, satisfying
$(x,\phi(x)) \in \cR$ for almost all $x \in D(\phi)$. Note that,
since $\cR$ is II$_1$, every $\phi \in [[\cR]]$ is measure
preserving. If $\Gamma \subset [[\cR]]$ is a subgroup, we denote by
$s(\Gamma)\subset X$ its {\it support}, i.e. the subset $Y\subset X$
with the property that $R(g)=D(g)=Y$, $\forall g\in \Gamma$. Two
such subgroups $\Gamma, \Lambda \subset [[\cR]]$ are {\it conjugate}
by an element in $[[\cR]]$ if there exists $\phi\in [[\cR]]$ such
that $D(\phi)=s(\Gamma)$, $R(\phi)=s(\Lambda)$ and $\phi \Gamma =
\Lambda \phi$.

\begin{theorem}\label{thm.countableT}
Let $\cR$ be a II$_1$ equivalence relation on the probability space
$(X,\mu)$.
\begin{enumerate}
\item If $[\cR]$ contains a property (T) group $\Gamma$ implementing an
ergodic action on $(X,\mu)$, then $\cF(\cR)$ is countable. More
generally, if $[\cR]$ contains a countable group $\Gamma$ having a
subgroup $H\subset \Gamma$ with the relative property (T)
implementing an ergodic action on $(X,\mu)$, then $\cF(\cR)$ is
countable.
\item
Let $\mathcal T$ be the set of property (T) subgroups $\Gamma
\subset [[\cR]]$ acting ergodically on $s(\Gamma)$. Then $\mathcal
T$ is countable, modulo conjugacy by elements in $[[\cR]]$.
\end{enumerate}
\end{theorem}

Note that the ergodicity assumption of the action of $\Gamma$ on
$(X,\mu)$ in Part 1 of the above statement is crucial. Indeed,
Theorem \ref{thm.Sfactorlarge} provides examples of free ergodic
p.m.p.\ actions $G \actson (X,\mu)$ such that $\cR(G \actson
X)$ has uncountable fundamental group, but nevertheless $G$ contains
a subgroup having property (T) (which, a fortiori, acts
non-ergodically on $(X,\mu)$). In turn, the existence of a property
(T) subgroup of $[\cR]$ acting freely and ergodically  on $X$, does
not ensure that the II$_1$ factor $L(\cR)$ has countable fundamental
group. Indeed, by \cite{CJ} there exist free ergodic p.m.p.\ actions
of groups of the form $G=\Gamma \times \Sigma$, with $\Gamma$ having
property (T) and acting by Bernoulli shifts (thus ergodically), such
that $M=\rL^\infty(X)\rtimes G$ splits off the hyperfinite II$_1$
factor, and thus $\cF(M)=\R_+$. In fact, as pointed out in
\cite{P3}, more than being countable, the fundamental group of
$\cR(G \actson X)$ is trivial.

Note also that in the case $\cR$ comes from a free ergodic action of
a property (T) group, $\Gamma \actson (X,\mu)$, Part 1 of Theorem
\ref{thm.countableT} was already shown in \cite[Corollary 1.8]{gg},
in the case $\Gamma$ is ICC, and in \cite[Theorem 5.9]{ioana}, in
the general case. We will use the above result in
\cite{PV-proc-Zimmer} to prove that the II$_1$ equivalence $\cR$
obtained by restricting the II$_\infty$ equivalence relation
implemented by $\SL(n,\Z) \curvearrowright \R^n$, $n \geq 4$,
to a subset of measure $1$ has property (T) in the sense of Zimmer,
yet cannot be implemented by an action (even non-free) of a property
(T) group because $\cF(\cR)=\R_+$.

We will prove Theorem \ref{thm.countableT} by contradiction, using the
property (T) of the subgroups and a ``separability'' argument, in
the spirit of \cite{Pcorr}. For more on this strategy of proofs,
which grew out of Connes' rigidity paper \cite{connes}, we send the
interested reader to Section 4 in \cite{Picm}. As a result of this
argument, we obtain two copies $\Gamma_1, \Gamma_2\subset [[\cR]]$
of the same property (T) group, which are uniformly close one to the
other. This in turn gives rise to a non-zero intertwiner $\phi \in
[[\cR]]$ between $\Gamma_1, \Gamma_2$. But if the $\Gamma_i$-actions
are assumed ergodic, this forces $\mu(s(\Gamma_1))=\mu(s(\Gamma_2))$
and the conjugacy of $\Gamma_1, \Gamma_2.$

The existence of an intertwiner between uniformly close subgroups in
$[[\cR]]$ is the subject of the next lemma. Recall that $[[\cR]]$
has a natural metric space structure, inherited from the
Hilbert-norm $\|\cdot \|_2$ of the underlying II$_1$ factor $L(\cR)$
associated with $\cR$, by viewing every $\phi \in [[\cR]]$ as a
partial isometry in $L(\cR)$ and using the $\|\cdot \|_2$-norm on
the latter. The metric can be concretely written as
$$d(\phi,\psi)^2 = \mu\bigl(D(\phi) \bigtriangleup\ D(\psi)\bigr) + 2
\mu\bigl(\{x \in D(\phi) \cap D(\psi) \mid \phi(x) \neq
\psi(x)\}\bigr) \; ,$$ where $\bigtriangleup$ denotes the symmetric
difference of two sets. We will also need the natural $\si$-finite
measure $\mu^{(1)}$ on $\cR \subset X \times X$, defined by the
formula
$$\mu^{(1)}(\cU) = \int_X \# \{y \mid (x,y) \in \cU \} \; d\mu(x) = \int_X
\#\{x \mid (x,y) \in \cU\} \; d\mu(y)$$ for all measurable subsets
$\cU \subset \cR$.

\begin{lemma} \label{lemma.conjugate-full}
Let $\cR$ be a II$_1$ equivalence relation on the standard probability space
$(X,\mu)$. Suppose that $\Gamma$ is a countable group, $X_0,Y_0 \subset X$ and let
$$\al : \Gamma \recht [\cR|_{X_0}] \quad\text{and}\quad \beta : \Gamma \recht [\cR|_{Y_0}]$$
be group morphisms satisfying $d(\al_g,\be_g) \leq 1/5$ for all $g \in
\Gamma$ and $\mu(X_0),\mu(Y_0) \geq 3/4$. Then there exist non-negligible measurable subsets $X_1
\subset X_0$, $Y_1 \subset Y_0$ and $\phi \in [[\cR]]$ with $D(\phi) =
X_1$, $R(\phi) = Y_1$ such that
\begin{align*}
& X_1 \;\;\text{is globally $(\al_g)_{g \in \Gamma}$-invariant}\;\; , \;\;
Y_1 \;\;\text{is globally $(\beta_g)_{g \in \Gamma}$-invariant, and}
\\ & \phi(\al_g(x)) = \beta_g(\phi(x)) \;\;\text{for almost all}\;\; x \in D(\phi) \; .
\end{align*}
\end{lemma}
\begin{proof}
Denote by $\Tr$ the normal faithful semi-finite trace on $\rL^\infty(\cR)$
given by integration along $\mu^{(1)}$. Let $p \in \rL^\infty(\cR)$ be the
projection onto $\cR \cap X_0 \times Y_0$ and $e \leq p$ the projection onto $\{(z,z)
\mid z \in X_0 \cap Y_0\}$. Set $B = \rL^\infty(\cR) p$. The group $\Gamma$ acts by
automorphisms $\rho_g$ of $B$ given by
$$(\rho_g F)(x,y) = F(\al_{g^{-1}}(x) , \be_{g^{-1}}(y)) \quad\text{for
almost all}\;\; (x,y) \in \cR \cap X \times Y \; .$$
Since $\|\rho_g(e) - e \|_{2,\Tr}^2 = 2 \mu\bigl( \{z \in X_0 \cap Y_0 \mid \al_{g^{-1}}(z)
\neq \be_{g^{-1}}(z)\}\bigr)$, we get
$$\|\rho_g(e) - e \|_{2,\Tr} \leq \frac{1}{5}\quad\text{for all}\;\; g \in \Gamma \; .$$
Define $a \in B^+$ as the unique element of minimal $\|\cdot\|_{2,\Tr}$
in the weakly closed convex hull $\overline{\operatorname{conv}}\{\rho_g(e)
\mid g \in \Gamma \}$. It follows that $\|a - e\|_{2,\Tr} \leq 1/5$ and that
$\rho_g(a) = a$ for all $g \in \Gamma$. Note that $0 \leq a \leq 1$. Defining
$f$ as the spectral projection $f=\chi_{[1/2,1]}(a)$, we find that $\|f-e\|_{2,\Tr}
\leq 2/5$ and $\rho_g(f) = f$ for all $g \in \Gamma$. We write $f = \chi_W$, where
$W \subset \cR \cap X_0 \times Y_0$ is globally $(\rho_g)_{g \in \Gamma}$-invariant and satisfies
\begin{equation}\label{eq.inequality}
\mu^{(1)}\bigl( W \; \bigtriangleup\ \{(z,z) \mid z \in X_0 \cap Y_0 \}\bigr) \leq \frac{4}{25} \; .
\end{equation}
Denote $_x W := \{y \in Y_0 \mid (x,y) \in W\}$ and $W_y := \{x \in X_0 \mid (x,y) \in W\}$. Define
$$W_0 := \{(x,y) \in W \mid \; _x W \;\;\text{and}\;\; W_y \;\;\text{are singletons}\;\} \; .$$
Then $W_0$ is still globally $(\rho_g)_{g \in \Gamma}$-invariant. Since $\mu(X_0),\mu(Y_0) \geq 3/4$ and $\mu(X_0 \bigtriangleup\ Y_0) \leq 1/25$,
we have $\mu(X_0 \cap Y_0) \geq 3/4 - 1/25$.
By \eqref{eq.inequality}, the set of $x \in X_0$
such that $_x W$ is a singleton then has measure at least $3/4 - 1/25 - 4/25$. The same holds for
the set of $y \in Y_0$ such that $W_y$ is a singleton. So, $W_0$ has measure at least $1/10$. By construction, $W_0$ is the graph of a partial automorphism
$\phi \in [[\cR]]$ satisfying all the conclusions of the lemma.
\end{proof}

\begin{proof}[Proof of Theorem \ref{thm.countableT}] Let us first
prove Part 1 of the theorem. By the relative property (T) of
$H\subset \Gamma$, there exist $F \subset \Gamma$ and $0 < \eps <
1/4$ such that whenever $\pi : \Gamma \recht \cU(\mathcal H)$ is a
unitary representation of $\Gamma$ on a Hilbert space $\mathcal H$
and $\xi_0 \in \mathcal H$ a unit vector satisfying $\|\pi(g) \xi_0
- \xi_0\| \leq \eps$ for all $g \in F$, then $\|\pi(h) \xi_0 -
\xi_0\| \leq 1/8$ for all $h \in H$.

Choose for every $t \in (0,1)$ a measurable subset $Y_t \subset X$
with $\mu(Y_t) = t$ and such that $Y_s \subset Y_t$ if $s \leq t$.
Assume that the fundamental group of $\cR$ is uncountable. For every
$t \in \cF(\cR) \cap (3/4,1)$, choose an isomorphism $\Delta_t : X
\recht Y_t$ between $\cR$ and $\cR|_{Y_t}$. Note that $\Delta_t$
scales the measure $\mu$ by $t$. Define $\al^t_g = \Delta_t \circ
\al_g \circ \Delta_t^{-1}$. Note that $\al^t_g \in [[\cR]]$ with
$D(\al^t_g) = R(\al^t_g) = Y_t$. Since $\cF(\cR) \cap (3/4,1)$ is
uncountable, separability of the metric space $([[\cR]],d)$ yields
$s,t \in \cF(\cR) \cap (3/4,1)$ with $s < t$ and $d(\al^s_g,\al^t_g)
\leq \eps/2$ for all $g \in F$.

Define the Hilbert space $\mathcal H = \rL^2(\cR \cap Y_s \times
Y_t,\mu^{(1)})$ and the unitary representation
$$\pi : \Gamma \recht \cU(\mathcal H) : (\pi(g)\xi)(x,y) = \xi(\al^s_{g^{-1}}(x),
\al^t_{g^{-1}}(y)) \; .$$
Set $\Delta_s := \{(y,y) \mid y \in Y_s\}$
and $\xi_0 := s^{-1/2} \chi_{\Delta_s}$. Then $\xi_0$ is a unit
vector in $\mathcal H$ and, for all $g \in F$,
$$\|\pi(g)\xi_0 - \xi_0\|^2 = 2s^{-1} \mu\bigl(\{y \in Y_s \mid \al^s_{g^{-1}}(y) \neq
\al^t_{g^{-1}}(y)\}\bigr) \leq s^{-1} d(\al^s_g,\al^t_g)^2 \leq \eps^2 \;
.$$ It follows that $\|\pi(h)\xi_0 - \xi_0\| \leq 1/8$ for all $h
\in H$. So, for all $h \in H$, we have
$$2 \mu\bigl(\{y \in Y_s \mid \al^s_h(y) \neq \al^t_h(y)\}\bigr) \leq
\frac{s}{64} \leq \frac{1}{64} \; .$$ Since also, given $g \in F$,
$$\mu(Y_t \setminus Y_s) \leq d(\al^s_g,\al^t_g)^2 \leq \frac{\eps^2}{4} \leq \frac{1}{64} \; ,$$
it follows that
$$d(\al^s_h,\al^t_h)^2 \leq \mu(Y_t \setminus Y_s) + \frac{1}{64} < \frac{1}{25}$$ for all
$h \in H$. Since $(\al_h)_{h \in H}$ implements an ergodic action on $(X,\mu)$,
the same holds for $(\al^s_h)_{h \in H}$, $(\al^t_h)_{h\in H}$ and
so, Lemma \ref{lemma.conjugate-full} provides an element $\phi \in
[[\cR]]$ with $D(\phi) = Y_s$ and $R(\phi) = Y_t$. Since $\phi$ is a
measure preserving isomorphism between $Y_s$ and $Y_t$ and
$\mu(Y_s)=s < t=\mu(Y_t)$, we reached a contradiction.

To prove Part 2 of Theorem \ref{thm.countableT}, assume by
contradiction that there exist uncountably many subgroups
$\{\Gamma_i\mid i\in I\}$ in $[[\cR]]$ which have property (T) and are non-conjugate in
$[[\cR]]$. We continue to use the measurable subsets $Y_t \subset X$ with $\mu(Y_t) = t$ and $Y_s \subset Y_t$ whenever $s \leq t$.
By the ergodicity of $\cR$, we may assume that for every $i \in I$, the support of $\Gamma_i$ is one of the $Y_s$.

By Shalom's theorem \cite[Theorem 6.7]{shalom}, every property (T) group is the quotient of a finitely presented property (T) group. Since there are only countably many finitely presented groups, we may assume that all $\Gamma_i$'s are quotients of the same property (T) group $\Gamma$ through surjective homomorphisms $\alpha_i : \Gamma \recht \Gamma_i$. Finally, we may assume that there exists $t \in (0,1)$ such that $\mu(s(\Gamma_i)) \in (3t/4,t)$ for all $i \in I$. So, replacing $\cR$ by $\cR|_{Y_t}$, we may assume that $\mu(s(\Gamma_i)) \in (3/4,1)$ for all $i \in I$.

By the property (T) of $\Gamma$, there exist $F \subset \Gamma$ and
$0 < \eps < 1/4$ such that whenever $\pi : \Gamma \recht
\cU(\mathcal H)$ is a unitary representation of $\Gamma$ on a
Hilbert space $\mathcal H$ and $\xi_0 \in \mathcal H$ a unit vector
satisfying $\|\pi(g) \xi_0 - \xi_0\| \leq \eps$ for all $g \in F$,
then $\|\pi(g) \xi_0 - \xi_0\| \leq 1/8 $ for all $g\in \Gamma$.

Now, by the separability of $([[\cR]], d)$, there exist $i\neq j$
such that $d(\alpha_i(g),\alpha_j(g))$ $\leq \varepsilon/2$, $\forall
g\in F$. Let $Y_i\subset X$, $Y_j\subset X$ be the support of
$\Gamma_i$ resp. $\Gamma_j$ and assume $Y_i\subset Y_j$. We define
$\mathcal H$, $\xi_0\in \mathcal H$, $\pi:\Gamma \rightarrow
\mathcal U(\mathcal H)$ as before, but replacing $\al^s_g$ by
$\al_i(g)$, $\al^t_g$ by $\al_j(g)$, $Y_s$ by $Y_i$ and $Y_t$ by
$Y_j$. The same estimates then show that $\xi_0$ is a unit vector
satisfying $\|\pi_g(\xi_0)-\xi_0\|\leq \eps$, $\forall g\in F$.
Thus, $\|\pi(g) \xi_0 - \xi_0\| \leq 1/8 $ for all $g\in \Gamma$. As
before, this translates into $d(\al_i(g),\al_j(g))\leq 1/4$,
$\forall g\in \Gamma$. By Lemma \ref{lemma.conjugate-full}, this
implies $\Gamma_i$, $\Gamma_j$ are conjugate by an element in
$[[\cR]]$, contradicting our initial assumption.
\end{proof}

Part 2 of Theorem \ref{thm.countableT} readily implies that the
functor $\Gamma \mapsto \cR_\Gamma$, from free ergodic p.m.p.\ actions
of property (T) groups with morphisms given by conjugacy, to
the associated equivalence relations with morphisms given by orbital
isomorphism, is ``countable to one''. In other words, there are at
most countably many non-conjugate free ergodic p.m.p.\ actions in
each OE class of a free ergodic p.m.p.\ action of a
property (T) group. In fact, even more is true: any free ergodic
p.m.p.\ action of a property (T) group follows ``orbit equivalent
superrigid, modulo countable classes'', in a sense made precise
below.

\begin{corollary}\label{superrig} Let $\Gamma \curvearrowright X$ be
a free ergodic p.m.p.\ action of a property (T) group. Let $\Lambda_i
\curvearrowright X_i$, $i\in I$, be a family of free ergodic p.m.p.\
actions such that $\cR_\Gamma \simeq \cR_{\Lambda_i}^{t_i}$, for some
$t_i > 0$. Then the family $I$ is countable, modulo conjugacy of
actions.
\end{corollary}
\begin{proof}
We may assume that $t_i \geq 1/c$ for all $i \in I$ and some $c > 0$. Setting $\cR = (\cR_\Gamma)^c$, we can view all $\Lambda_i$ as subgroups of $[[\cR]]$, with the action of $\Lambda_i$ on $s(\Lambda_i) \subset X$ being conjugate to $\Lambda_i \actson X_i$. By \cite[Corollary 1.4]{Fu}, all $\Lambda_i$ have property (T).
So, by Part 2 of Theorem \ref{thm.countableT}, the family $I$ is countable modulo conjugacy of actions.
\end{proof}

When $\Gamma$ is an ICC property (T) group, all groups in
$\Sfactor(\Gamma)$ are countable (cf.\ \cite[Proof of Theorem
1.7]{gg}, or \cite[Theorem 4.5.1]{Pcorr}). The next theorem
generalizes this result to Kazhdan groups $\Gamma$ with the property
that the center $\cZ(\Gamma)$ has finite index in the
\emph{FC-radical} $\Gamma_f$ of $\Gamma$, defined by
$$\Gamma_f := \{g \in \Gamma \mid \; g \;\;\text{has a finite conjugacy class}\;\} \; .$$
On the other hand, we prove in the second part of the
theorem below that if $\Gamma$ is a residually finite property (T)
group such that $\Gamma_f$ is not virtually abelian (i.e.,
$\cZ(\Gamma_f) < \Gamma_f$ has infinite index), then $\Gamma$ admits
a free ergodic p.m.p.\ action on $(X,\mu)$ with $\rL^\infty(X)
\rtimes \Gamma$ being McDuff and hence, $\R_+ \in \Sfactor(\Gamma)$.

At the time of finishing a first version of this article, the only known examples of Kazhdan groups $\Gamma$ with
infinite FC-radical $\Gamma_f$ were such
that $\cZ(\Gamma)$ has finite index in $\Gamma_f$ (see e.g.\
\cite[Example 1.7.13]{BHV} and \cite[Definition 2.4]{cornulier}).
While we were unable to show whether or not there exist residually finite Kazhdan groups $\Gamma$ with non virtually abelian FC-radical $\Gamma_f$, after discussing this problem with several
specialists, it was indicated to us by Mark Sapir and Denis Osin
that such groups probably do exist. Very recently, this was confirmed by Mikhail Ershov \cite{Er2} who showed that every Golod-Shafarevich group has a residually finite quotient whose FC-radical is not virtually abelian. Since he proved in \cite{Er1} that there exist Golod-Shafarevich groups with property (T) and since property (T) passes to quotients, there indeed exist free ergodic p.m.p.\ actions $\Gamma \actson (X,\mu)$ of property (T) groups such that $\rL^\infty(X) \rtimes \Gamma$ is McDuff.

\begin{theorem}\label{thm.TMcDuff}
Let $\Gamma$ be a property (T) group.
\begin{enumerate}
\item If $\cZ(\Gamma)$ has finite index in the FC-radical $\Gamma_f$, then $\Sfactor(\Gamma)$ only contains countable groups.
\item If $\Gamma$ is residually finite and $[\Gamma_f,\cZ(\Gamma_f)] = \infty$, then $\Gamma$ admits a free ergodic profinite p.m.p.\ action on $(X,\mu)$ such that $\rL^\infty(X) \rtimes \Gamma$ is McDuff.
\end{enumerate}
\end{theorem}
\begin{proof}
Whenever $H \subset \Gamma$, denote by $\rC_\Gamma(H)$ the centralizer of $H$ inside $\Gamma$.

Assume first that $\cZ(\Gamma)$ has finite index in $\Gamma_f$. Let $\Gamma \actson (X,\mu)$ be free ergodic p.m.p. Write $A := \rL^\infty(X)$ and $M := A \rtimes \Gamma$. Define $\Lambda := \rC_\Gamma(\Gamma_f)$. Since $\cZ(\Gamma)$ has finite index in $\Gamma_f$, it follows that $\Lambda$ has finite index in $\Gamma$. A fortiori, the subgroup $\Lambda_1 := \Lambda \cdot \Gamma_f$ has finite index in $\Gamma$. Also, the subalgebra $A^\Lambda$ of $\Lambda$-invariant functions in $A$, is finite dimensional and globally $\Lambda_1$-invariant. Consider the subalgebra $B := A^\Lambda \rtimes \Lambda_1$ of $M$. Since $L(\Lambda_1) \subset B$ has finite index, it follows that $B$ has property (T). On the other hand $M \cap B' \subset M \cap L(\Lambda)'$ and it is straightforward to check that $M \cap L(\Lambda)' \subset A^\Lambda \rtimes \Gamma_f$. So, we get $M \cap B' \subset B$. By \cite[Theorem A.1]{NPS}, it follows that $\cF(M)$ is countable.

Suppose from now on that $\Gamma$ is residually finite and $[\Gamma_f,\cZ(\Gamma_f)] = \infty$. Let $\Gamma_f = \{ h_1,h_2,\ldots \}$ be an enumeration. Let $H_n \lhd \Gamma$ be a decreasing sequence of normal, finite index subgroups with $\bigcap_n H_n = \{e\}$. Define
$$\Gamma_n := H_n \cap \bigcap_{s \in \Gamma / \rC_\Gamma(h_1,\ldots,h_n)} s \rC_\Gamma(h_1,\ldots,h_n) s^{-1} \; .$$
By construction, $\Gamma_n$ is a decreasing sequence of normal, finite index subgroups with $\bigcap_n \Gamma_n = \{e\}$ and such that for all $h \in \Gamma_f$, we have $\Gamma_n \subset \rC_\Gamma(h)$ for all $n$ large enough.

Denote $(X,\mu) = \invlimit (\Gamma/\Gamma_n , \text{counting probability measure})$. Consider the natural free, ergodic, profinite, p.m.p.\ action $\Gamma \actson (X,\mu)$. Put $A = \rL^\infty(X)$ and $M := A \rtimes \Gamma$. For every $s \in \Gamma$ and $n \in \N$, denote by $\chi_{s \Gamma_n}$ the function equal to $1$ on $s \Gamma_n$ and zero elsewhere and interpret $\chi_{s \Gamma_n}$ as a projection in $A$.

For every $h \in \Gamma_f$, define the unitary $v_h \in M \cap \rL(\Gamma)'$ by
$$v_h := \sum_{s \in \Gamma / \Gamma_n} \chi_{s \Gamma_n} u_{s h^{-1} s^{-1}} \quad\text{for $n$ large enough, meaning $\Gamma_n \subset \rC_\Gamma(h)$.}$$
It is straightforward to check that $\Gamma_f \recht \cU(M \cap \rL(\Gamma)') : h \mapsto v_h$ is a group morphism and that $\tau(v_h) = 0$ whenever $h \neq e$.

{\bf Claim.} If for all $n \in \N$, we have $h_n \in \Gamma_f \cap \Gamma_n$ with $h_n \neq e$, then $(v_{h_n})$ is a central sequence in $M$ with $\tau(v_{h_n}) = 0$ for all $n$. For all $n$, we have $v_{h_n} \in \rL(\Gamma)'$. So, to prove the claim, it suffices to take $k \in \N$, $g \in \Gamma$ and prove that
$$\lim_n \|[\chi_{g \Gamma_k} , v_{h_n}]\|_2 = 0 \; .$$
But, by construction, $\chi_{g \Gamma_k}$ and $v_{h_n}$ commute when $n \geq k$.

Since $\cZ(\Gamma_f) < \Gamma_f$ has infinite index and since $\Gamma_f$ has finite conjugacy classes, it follows that $\Gamma_f$ has no finite index abelian subgroups. So, for every $n$, the finite index subgroup $\Gamma_f \cap \Gamma_n$ of $\Gamma_f$ is non-abelian. Therefore, we can choose $h_n,h'_n \in \Gamma_f \cap \Gamma_n$ such that $h_n h'_n h_n^{-1} {h'_n}^{-1} \neq e$. By the claim above, $v_{h_n}$ and $v_{h'_n}$ are central sequences. By construction, $\tau(v_{h_n} v_{h'_n} v_{h_n}^* v_{h'_n}^*) = 0$ for all $n$. So, $M$ is McDuff.

\end{proof}

\end{document}